\documentclass[
final,
 nomarks,
]{dmtcs-episciences}

\usepackage[utf8]{inputenc}
\usepackage{subfigure}
\usepackage{scrextend}
\usepackage{amsmath}
\usepackage{graphicx}
\usepackage{amssymb}
\usepackage{amsthm}

\usepackage{amsmath,amsthm,color,eepic}
\usepackage{amssymb} \usepackage{verbatim}
\usepackage{graphicx}
\usepackage{float}

\usepackage{hyperref}
\usepackage{cite}

\usepackage{graphicx}

\usepackage{amsmath,amsthm,color,eepic}
\usepackage{amssymb} 
\usepackage{verbatim}
\usepackage{graphicx}
\usepackage{float}

\usepackage{enumitem}

\newtheorem{theorem}{Theorem}[]
\newtheorem{proposition}[theorem]{Proposition}
\newtheorem{lemma}[theorem]{Lemma}

\theoremstyle{definition}

\usepackage[square,numbers]{natbib}

\author{Kasper Szabo Lyngsie\affiliationmark{1}}
\title{On neighbour sum-distinguishing \{0,1\}-weightings of bipartite graphs}
\affiliation{
  Technical University of Denmark, Denmark}
\received{2017-1-5}
\revised{2018-4-16}
\accepted{2018-5-24}
\begin{document}
\publicationdetails{20}{2018}{1}{21}{2632}
\maketitle
\begin{abstract}
   Let $S \subset \mathbb{Z}$ be a set of integers. A graph $G$ is said to have the \emph{$S$-property} if there exists an $S$-edge-weighting $w: E(G) \rightarrow S$ such that any two adjacent vertices have different sums of incident edge-weights. In this paper we characterise all bridgeless bipartite graphs and all trees without the \{0,1\}-property. In particular this problem belongs to \textsf{P} for these graphs while it is \textsf{NP}-complete for all graphs. 
\end{abstract}

\section{Introduction}
The problems investigated in this paper are highly related to the well-known \emph{1,2,3-Conjecture} formulated in~\cite{KaLuTh}. One way to approach this conjecture (see for example \cite{kha}) has been to study the \textit{$\{a,b\}$-property} of graphs for two integers $a$ and $b$ defined in the following way: a graph $G$ is said to have the $\{a,b\}$-property if there exists a mapping $w: E(G) \rightarrow \{a,b\}$ such that for all pairs of adjacent vertices $u$ and $v$ we have $\sum_{e \in E(v)}w(e) \neq \sum_{e \in E(u)}w(e)$, where $E(v)$ and $E(u)$ denote the edges incident to $v$ and $u$ respectively. We call $w$ a \textit{neighbour sum-distinguishing edge-weighting} of $G$ with weights $a$ and $b$.  \\

\noindent In~\cite{Lu} Lu investigated the problem of determining whether or not a given bipartite graph has the $\{0,1\}$- or the $\{1,2\}$-property.  The restriction to bipartite graphs was motivated by a result by Dudek and Wajc~\cite{DuWa} saying that the problem is \textsf{NP}-complete for general graphs. In particular Lu asked the natural question whether the problem is polynomial if only bipartite graphs are considered (Problem 1 in \cite{Lu}). The results of the present paper answer in the affirmative for bridgeless bipartite graphs and trees. Lu also proved the following theorem:
\begin{theorem}\emph{~\cite{Lu}} \label{thm:LU}
Every $2$-connected and $3$-edge-connected bipartite graph has the $\{0,1\}$- and the $\{1,2\}$-property.
\end{theorem}
\noindent In \cite{skow} Skowronek-Kazi\'{o}w investigated the problem of determining whether a graph has a $\{1,2\}$-edge-weighting such that the following vertex-colouring is proper: for each vertex $v$, assign the product of the edge-weights incident to $v$ as $v$'s colour. This product-property is the same as the $\{0,1\}$-property and Skowronek-Kazi\'{o}w verified this for various classes of bipartite graphs, for example bipartite graphs of minimum degree at least 3. In \cite{skow} Skowronek-Kazi\'{o}w also asked for a characterization of all bipartite graphs, in particular trees, which have such $\{1,2\}$-edge-weightings, that is, which have the $\{0,1\}$-property. As mentioned above the results of the present paper give such a characterization for trees and bridgeless bipartite graphs.\\
\noindent  A bipartite graph without the \{0,1\}-property is said to be \textit{bad}.  \\
Thomassen, Wu and Zhang~\cite{TWZ} gave a complete characterisation of all bipartite graphs without the $\{1,2\}$-property. Any such graph is an \textit{odd multi-cactus} defined as follows:
Take a collection of cycles of length 2 modulo 4, each of which have edges coloured alternately red and green. Then form a connected simple graph by pasting the cycles together, one by one, in a tree-like fashion along green edges. Finally replace every green edge by a multiple edge of any multiplicity. The graph with one edge and two vertices is also called an odd multi-cactus. It can easily be checked that an odd multi-cactus do not have the $\{a,b\}$-property for any $a,b \in \mathbb{Z}$. As mentioned above these graphs characterise the bipartite graphs without the \{1,2\}-property:
\begin{theorem} \label{thm:12oddmul} \emph{~\cite{TWZ}}
$G$ is a connected bipartite graph without the $\{1,2\}$-property if and only if $G$ is an odd multi-cactus.
\end{theorem}
\noindent Since an odd multi-cactus is recognisable in polynomial-time, this answers the part of Lu's problem from~\cite{Lu} concerning the $\{1,2\}$-property. As pointed out in~\cite{TWZ}, Theorem~\ref{thm:12oddmul} extends to all positive edge-weights $a$ and $b$ of distinct parity but not to the edge-weights $0$ and $1$. In~\cite{TWZ} it is also remarked that any bipartite graph of minimum degree at least 3 has the $\{a,b\}$-property for all pairs of non-negative integers $a,b$ of distinct parity. Thus it remains open to characterise those bipartite graphs with cut-vertices and minimum degree at most 2 which do not have the \{0,1\}-property. \\  
In~\cite{Lu}, Lu gave the following example of a bad graph with the $\{1,2\}$-property: Two 6-cycles connected by a path of length 3 and, as noted in~\cite{TWZ}, we can construct an infinite number bad graphs with the $\{1,2\}$-property by the following procedure: Take two graphs without the $\{0,1\}$-property and join them by a path of length 3 modulo 4. We can even generalise this procedure further: Let $s \geq 0$ be an integer and let $P$ be a path of length 1 modulo 4. Join each intermediate vertex in $P$ to $s$ bad graphs by $s$ edges, and join the end-vertices of $P$ to $s+1$ bad graphs (see Figure~\ref{fig:badex1}). This will create a new bad graph with the $\{1,2\}$-property. 
\begin{figure}[H]
\centering
\includegraphics[scale=0.8]{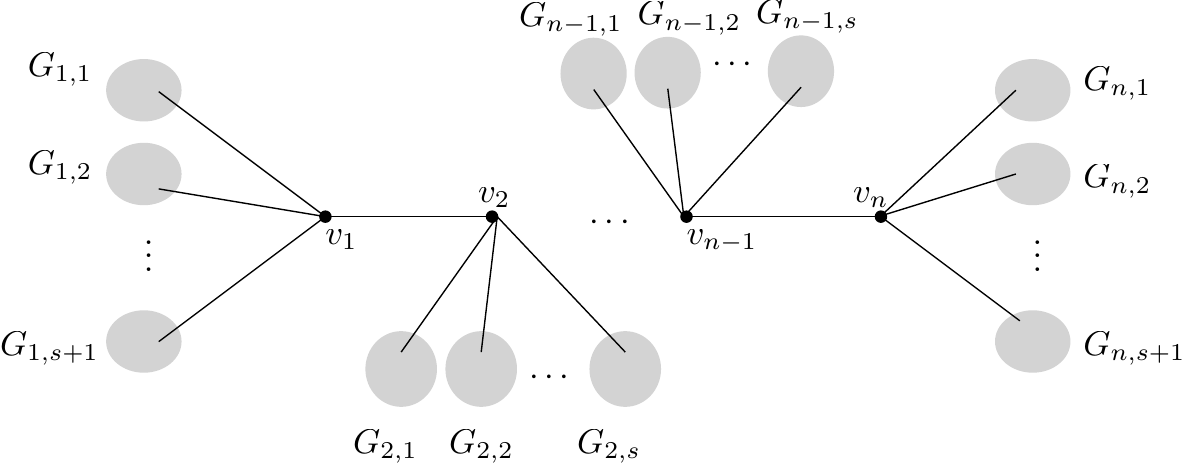} 
\caption{A construction of bad graphs with the \{1,2\}-property.}
\label{fig:badex1}
\end{figure}

\noindent Although the preceding paragraph shows a large class of bipartite graphs without the \{0,1\}-property, the list is still not complete. Not even for trees, as demonstrated by the tree in Figure~\ref{fig:smallbadtree}.
\begin{figure}[H] 
\centering 
\includegraphics[scale=0.8]{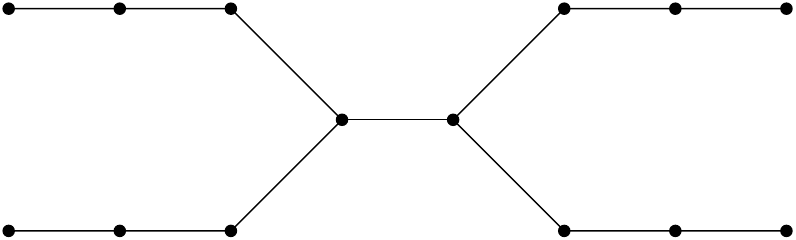} 
\caption{A tree without the \{0,1\}-property.}
\label{fig:smallbadtree}
\end{figure} 
\noindent Thus there is a large class of bad graphs which are not odd multi-cacti and it seems that the \{0,1\}-property is very different from the \{1,2\}-property. However, note that the above procedure always create bridges. This gives the hint that the \{0,1\}-property and the \{1,2\}-property might behave in a similar way if we don't allow bridges. This is indeed true as we prove in Section~\ref{sec1}: 
\begin{theorem} \label{Thm:1}
$G$ is a connected bridgeless bipartite graph without the $\{0,1\}$-property if and only if $G$ is an odd multi-cactus.
\end{theorem}
\noindent As mentioned after Theorem~\ref{thm:12oddmul}, an odd multi-cactus is recognisable in polynomial time so this answers the part of Lu's problem from~\cite{Lu} concerning the \{0,1\}-property for bridgeless bipartite graphs. In Section~\ref{sec2} we provide additional operations for constructing trees without the $\{0,1\}$-property. The class of trees without the $\{0,1\}$-property we can obtain using these operations we call $\mathcal{B}$ and these are all recognisable in polynomial time. Whilst the class $\mathcal{B}$ is difficult to describe we show that this gives a full characterisation of all bad trees.
\begin{theorem} \label{Thm:2}
A tree $T$ has the \{0,1\}-property unless $T$ is a member of $\mathcal{B}$.
\end{theorem}
\noindent Taken together, Theorems~\ref{Thm:1} and~\ref{Thm:2} show a marked difference between the \{0,1\}-problem and the \{1,2\}-problem. Indeed for bridgeless bipartite graphs Theorems~\ref{Thm:1} and~\ref{thm:12oddmul} show that the class of graphs without the \{0,1\}-property and the class of graphs without the \{1,2\}-property are precisely the same. On the other hand, Theorems~\ref{thm:12oddmul} and~\ref{Thm:2}  show that this is far from the case with trees.

\section{Bridgeless bipartite graphs without the \{0,1\}-property} 
\label{sec1}
Let $G$ be a bipartite graph. A \emph{$\{0,1\}$-weighting} of $G$ is a map $w: E(G) \rightarrow \{0,1\}$. Given a $\{0,1\}$-weighting $w$ of $G$ and a vertex $v$ of $G$ we call the sum $\sum_{e \in E(v)}w(e)$ the \textit{weighted degree} of $v$ or the induced \textit{colour} of $v$ (induced by $w$). For convenience the weighted degree of a vertex $v$ is also denoted $w(v)$. We say that a $\{0,1\}$-weighting $w$ is \textit{neighbour sum-distinguishing} or \textit{proper} if for all pairs of adjacent vertices $u, v$ it holds that $\sum_{e \in E(v)}w(e) \neq \sum_{e \in E(u)}w(e)$. That is, if the induced vertex-colouring is proper. If $w$ is a $\{0,1\}$-weighting of $G$ and two adjacent vertices $u$ and $v$ have the same weighted degree, then we say that the edge $uv$ is a \textit{conflict}. If two adjacent vertices $u,v$ have the same weighted degree parity we call the edge $uv$ a \textit{parity conflict}. Note that a parity conflict is not necessarily a conflict. If $f : V(G) \rightarrow \mathbb{Z}_k$ is a mapping and $H$ is a spanning subgraph of $G$ such that for all vertices $v$ we have $d_H(v) \equiv f(v) \mod k$ then we say that $H$ is an \textit{f-factor modulo k}. These factors play an important role in the investigations of $\{a,b\}$-properties for bipartite graphs, in particular because of the following result mentioned in~\cite{TWZ}:
\begin{lemma} \label{lem:f-factor}\emph{~\cite{TWZ}}
Let $G$ be a connected graph. If $f: V(G) \rightarrow \mathbb{Z}_2$ is a mapping satisfying \\ $\sum_{v \in V(G)} f(v) \equiv 0 \mod 2 $, then $G$ contains an \textit{f-factor} modulo $2$.
\end{lemma}
\noindent As also pointed out in \cite{kha}, \cite{skow} and \cite{TWZ} this immediately implies that all bipartite graphs where one bipartition set has even size have the $\{a,b\}$-property when $a$ and $b$ are numbers of different parity, since the weighted degree of all the vertices belonging to the even-sized bipartition set can get odd weighted degree while all other vertices get even weighted degree. So the problem is reduced to the case where both bipartition sets have odd size. Another useful tool is Lemma~\ref{lem:remove-edges} below.
\begin{lemma} \label{lem:remove-edges}\emph{~\cite{TWZ}}
Let $q$ be a natural number such that  $ q \geq 4$. Let $G$ be a connected graph and let $A$ be an independent set of at most $q$ vertices such that each vertex in $A$ has degree at least $q-1$, or, each vertex in $A$, except possibly one has degree at least $q$. Assume that no vertex in $A$ is adjacent to a bridge in $G$. Then, for each vertex $a$ of $A$, there is an edge $e_a$ incident with $a$ such that the deletion of all $e_a$, $a \in A$, results in a connected graph unless $|A|=q=4$, all vertices of $A$ have degree $3$ and $G-A$ has six components each of which is joined to two distinct vertices of $A$.
\end{lemma}
\noindent As can be seen in~\cite{TWZ} and later in this paper, Lemma~\ref{lem:f-factor} and~\ref{lem:remove-edges} work well together under some assumptions in the following way: Let $G$ be a simple bipartite graph with an odd number of vertices in both bipartition sets $X$ and $Y$, and let $w_0$ be a vertex belonging to $X$ with at least 4 neighbours and which is not a cutvertex. Assume that no neighbour of $w_0$ has greater degree than $w_0$ (such a vertex $w_0$ is said to have \emph{local maximum degree}), and such that no neighbour of $w_0$ having the same degree as $w_0$ is incident to a bridge in $G-w_0$. Furthermore, assume that we are not in the exceptional case in Lemma~\ref{lem:remove-edges} when we remove $w_0$ and choose $A$ to be the neighbours of $w_0$ having the same degree as $w_0$. Now we can find a proper $\{0,1\}$-weighting of $G$ as follows. We remove $w_0$ and an edge $e_a$ incident to each $a \in A$ and maintain connectivity by Lemma~\ref{lem:remove-edges}. We call the resulting graph $G'$. First consider the case where $w_0$ has even degree. By Lemma~\ref{lem:f-factor} we find a $\{0,1\}$-weighting of $G'$ such that all vertices in $X\backslash \{w_0\}\cup N(w_0)$ have odd weighted degree and all vertices in $Y\backslash N(w_0)$ have even weighted degree. Now we extend this $\{0,1\}$-weighting to the whole of $G$ by assigning weight 1 to all edges incident to $w_0$ and weight 0 to all edges $e_a$. The parity conflicts are between $w_0$ and its neighbours, but because all edges $e_a$ have weight 0, the weighted degree of $w_0$ is strictly greater than that of all its neighbours.
In the case where $w_0$ has odd degree, we find a $\{0,1\}$-weighting of $G'$ such that all vertices in $X\backslash \{w_0\} \cup N(w_0)$ have even weighted degree and all vertices in $Y\backslash N(w_0)$ have odd weighted degree. As before we extend this $\{0,1\}$-weighting to the whole of $G$ by assigning weight 1 to all edges incident to $w_0$ and weight 0 to all edges $e_a$. \\
Note that this shows that whenever we consider a vertex $w_0$ which is not a cutvertex, then we can find a $\{0,1\}$-weighting where all edges incident to $w_0$ have weight 1 and the only parity conflicts are between $w_0$ and its neighbours. \\ \\

\noindent Before we prove Theorem~\ref{Thm:1}, we will need three facts about simple odd multi-cacti formulated in Lemmas~\ref{lem:1},~\ref{lem:2} and~\ref{lem:3} below. \\
If $G$ is an odd multi-cactus then, by definition, $G$ contains at least two cycles containing two adjacent vertices with at least three neighbours each in $G$ while the other vertices all have two neighbours in $G$, unless $G$ is a single cycle or $K_2$ (possibly with multiple edges). Cycles of this type are called \textit{end-cycles} in $G$.  
\begin{lemma} \label{lem:1}
Let $G \neq K_2$ be a simple odd multi-cactus. For any vertex $v \in V(G)$ there is a $\{0,1\}$-weighting of $G$ such that $v$ and all vertices in the opposite bipartition set to $v$ get weighted degree $1$ and all other vertices get weighted degree $0$ or $2$.
\end{lemma}
\begin{proof}
The proof is by induction on the number of vertices $n$. It is easy to check that the statement is true for a single cycle of length 2 modulo 4, so assume $n>6$. Let $C$ be an end-cycle in $G$ such that $v$ is not a vertex in $C$ with only two neighbours. We can assume $C$ is a 6-cycles since subdividing edges with four vertices preserves the conclusion of the lemma. Thus, say that $C=v_1v_2 \cdots v_6v_1$, where $v_1$ and $v_2$ have at least three neighbours in $G$. Since $v$ is in $G-\{v_3,v_4,v_5,v_6\}$ we can use the induction hypothesis on $G- \{v_3, v_4, v_5, v_6\}$ and extend this $\{0,1\}$-weighting to the whole of $G$. 
\end{proof}

\noindent Let $w$ be a $\{0,1\}$-weighting of $G$, let $v$ be a vertex of $G$ and let $a$ be a natural number. Finally, let $C_w$ denote the vertex-colouring induced by $w$. Let $C_w(v,a)$ denote the colouring obtained from $C_w$ by replacing the colour $C_w(v)$ of $v$ with the colour $C_w(v)+a$. If $C_w(v,a)$ is a proper vertex-colouring we say that \textit{$w$ is a proper $\{0,1\}$-weighting of $G$ when the degree of $v$ is increased by $a$}. This may be thought of as a neighbour sum-distinguishing edge-weighting where the vertex $v$ has some pre-assigned weight.

\begin{lemma} \label{lem:2}
Let $G \neq K_2$ be a simple odd multi-cactus. Furthermore, let $u,v$ be any two vertices in $G$ belonging to the same bipartition set (possibly $u=v$). There is a proper $\{0,1\}$-weighting of $G$ when the weighted degrees of both $u$ and $v$ are increased by $1$ (if $u=v$ the weighted degree is increased by $2$).
\end{lemma}
\begin{proof}
First note that the case $u=v$ follows from Lemma~\ref{lem:1}, so we assume that $u \neq v$.  The proof is by induction on the number of vertices $n$. It is easy to check that the statement holds for a single cycle of length 2 modulo 4. As in the proof of Lemma~\ref{lem:1} we choose and end-cycle $C$ such that one of $v$ and $u$, say, $u$ is not a vertex in $C$ with only two neighbours in $G$ and we may assume that $C=v_1v_2 \cdots v_6v_1$, where $v_1$ and $v_2$ have at least three neighbours in $G$. If $v$ and $u$ are both in $G- \{v_3, v_4, v_5, v_6\}$ then we use the induction hypothesis on $G- \{v_3, v_4, v_5, v_6\}$ and get a proper $\{0,1\}$-weighting of $G-\{v_3, v_4, v_5, v_6\}$ if the weighted degree of both $u$ and $v$ are increased by 1. We can easily extend this $\{0,1\}$-weighting to the whole of $G$, a contradiction. So we can assume that $u$ is in $G- \{v_3, v_4, v_5, v_6\}$ and $v$ is one of $v_3,v_5$ (the other cases are similar). If $u$ is one of $v_1, v_2$, say, $v_1$ and $v$ is $v_5$, then we use Lemma~\ref{lem:1} on $G- \{v_3, v_4, v_5, v_6\}$ choosing $v_1$ as our special vertex. Then we get a $\{0,1\}$-weighting $w$ of $G- \{v_3, v_4, v_5, v_6\}$ where $v_1$ and all vertices in the opposite bipartition set to $v_1$ get weight 1 and all other vertices get weight 0 or 2. We extend this $\{0,1\}$-weighting to the whole of $G$ by defining $w(v_1v_6)=w(v_4v_5)=1$ and $w(v_2v_3)=w(v_3v_4)=w(v_5v_6)=0$. \\
If $u=v_1$ and $v$ is $v_3$, then again we use Lemma~\ref{lem:1} on $G- \{v_3, v_4, v_5, v_6\}$ choosing $v_1$ as our special vertex. As before we get a $\{0,1\}$-weighting $w$ of $G- \{v_3, v_4, v_5, v_6\}$ we can extend to the whole of $G$ by defining $w(v_1v_6)=w(v_3v_4)=1$ and $w(v_2v_3)=w(v_4v_5)=w(v_5v_6)=0$.  \\
This leaves us with the case where $u$ is in $G-C$ and $v$ is one of $\{v_3, v_4, v_5, v_6\}$. We can assume that $v_1$ is in the same bipartition set as $u$ and we start by considering the case where $v=v_3$. In this case we use the induction hypothesis on $G- \{v_3, v_4, v_5, v_6\}$ choosing $u$ and $v_1$ as our special vertices. We extend this $\{0,1\}$-weighting, letting the edge $v_1v_6$ play the role of the extra weight on $v_1$ by defining $w(v_1v_6)=1$ and $w(v_2v_3)=0$. Now $v_1$ and $v_2$ have different weighted degrees by the induction hypothesis so we can choose the weights on $v_3v_4$ and $v_5v_6$ to be different such that we avoid conflicts between $v_6$ and $v_1$, between $v_3$ and $v_2$ and between $v_4$ and $v_5$. Finally we define $w(v_4v_5)=0$ to avoid conflicts between $v_5$ and $v_6$, and between $v_3$ and $v_4$. \\
The case where $v=v_5$ remains. Here we use Lemma~\ref{lem:1} on  $G- \{v_3, v_4, v_5, v_6\}$ choosing $u$ as our special vertex and extend this $\{0,1\}$-weighting to $G$ by defining $w(v_2v_3)=w(v_1v_6)=w(v_3v_4)=0$ and $w(v_5v_6)=w(v_4v_5)=1$. 
\end{proof}

\noindent Using Lemma~\ref{lem:1} and induction as in the proof of Lemma~\ref{lem:1} we can easily derive Lemma~\ref{lem:3} below.
\begin{lemma} \label{lem:3}
Let $G$ be an odd multi-cactus where the red-green edge-colouring is unique. If $G'$ is obtained from $G$ by replacing a red edge with an edge of multiplicity $> 1$, then $G$ has the $\{0,1\}$-property.
\end{lemma}
\noindent In a graph $G$, a \textit{suspended path} or \textit{suspended cycle} is a path or cycle $v_1v_2...v_q$ such that all intermediate vertices have degree 2 and the end-vertices $v_1,v_q$ have degree at least 3. All vertices $v_1,...,v_q$ should be distinct, except that possibly $v_1=v_q$ (if it is a suspended cycle). \\
Having these small facts established we are ready for the proof of Theorem~\ref{Thm:1}. The proof follows the same approach as the proof of Theorem~\ref{thm:12oddmul} in~\cite{TWZ}, but new problems arise which have to be dealt with along the way. At the end of the proof, the reader is referred to~\cite{TWZ}.

\begin{proof}[of Theorem~\ref{Thm:1}]
It suffices to prove that if $G$ is a connected bridgeless bipartite graph without the $\{0,1\}$-property, then $G$ is an odd multi-cactus. \\
Suppose the theorem is false and let $G$ be a smallest counterexample. That is, among all bridgeless bipartite graphs without the \{0,1\}-property which are not odd multi-cacti, $G$ has the fewest vertices and subject to that, the fewest edges. Note that by induction and Lemma~\ref{lem:3} we can assume that if there is an edge $uv$ of multiplicity greater than 1, then $uv$ must have multiplicity 2 and be a bridge in the simple graph underlying $G$. Let $X$ and $Y$ be the two bipartition sets of $G$. By the remark following Lemma~\ref{lem:f-factor} we can assume that both $X$ and $Y$ have odd size. \\
First note that if $v \in X$ is a vertex in $G$ which is only adjacent to one other vertex $v'$ (since $G$ is bridgeless the multiplicity of $uv$ is then at least 2), then for any edge $e=v'v'' \neq v'v$ in $G-v$ incident to $v'$, the graph $G'=G-v-e$ is connected. So by Lemma \ref{lem:f-factor} the graph $G'$ contains a spanning subgraph $H$ where all vertices in $(X \setminus \{v,v''\}) \cup \{v'\}$ have odd degree and vertices in $Y \setminus \{v'\}$ have even weighted degree. By assigning weight 1 to all edges in $E(H) \cup \{e\}$ and weight 0 to all other edges we get a proper $\{0,1\}$-weighting of $G$, a contradiction. Thus, we can assume that there is no vertex $v$ in $G$ which is only adjacent to one other vertex $v'$.  \\ 
Let $B$ denote an endblock in $G$. Note that the above implies that there are no multiple edges in $B$. \\ \\
\textbf{Claim 1:} $B$ contains no suspended path of length 2. \\ \\
Assume that $y_1xy_2$ is a suspended path of length 2 in $B$, where $x \in X$ and $y_1,y_2 \in Y$. By Lemma~\ref{lem:f-factor} there exists a spanning subgraph $H$ of $G'=G-x$ such that all vertices in $X\backslash \{x\}$ have odd degree and all vertices in $Y$ have even degree. From $H$ we can construct a $\{0,1\}$-weighting $w_{G'}$ of $G'$ such that each vertex in $X\backslash \{x\}$ has odd weighted degree and each vertex in $Y$ has even weighted degree. We do this by assigning weight 1 to the edges in $H$ and weight 0 to the edges outside $H$. We extend this $\{0,1\}$-weighting to a $\{0,1\}$-weighting $w_{G}$ of the whole graph $G$ by defining $w_G(xy_1)=w_G(xy_2)=0$. The only possible conflicts are $xy_1$ and $xy_2$ in the case where $w_{G'}(y_1)=0$ or $w_{G'}(y_2)=0$. If we can remove an edge $e_1=y_1z_1$ incident to $y_1$ and an edge $e_2=y_2z_2$ incident to $y_2$ in $G'$ and still have a connected graph, then we can avoid this situation as follows: using Lemma~\ref{lem:f-factor} we redefine $H$ to be a subgraph of $G-x-e_1-e_2$ such that all vertices in $X-x-z_1-z_2 \cup \{y_1,y_2\}$ have odd weighted degree and all other vertices have even weighted degree. Then define $w_{G'}$ to be the $\{0,1\}$-weighting assigning weight 1 to all edges in $E(H) \cup \{e_1, e_2 \}$ and weight 0 to all other edges. This is a proper $\{0,1\}$-weighting of $G$, so we can assume that we cannot remove two edges incident to $y_1$ and $y_2$ respectively in $G'$ and still have a connected graph.  \\ 
There must be a cycle, $C$, going through $y_1$ and $y_2$ in $G'$ since otherwise $y_1$ and $y_2$ lie in distinct blocks of $G'$, and since the degree of both $y_1$ and $y_2$ is at least 3 and since $G$ is bridgeless it is now possible to remove an edge from both $y_1$ and $y_2$ and still have a connected graph, a contradiction.  We first look at the case where $w_{G'}(y_1)=w_{G'}(y_2)=0$. Here we swap all the weights in $C$ (that is, we change all 1-weights to 0 weights and all 0-weights to 1-weights). This will not change the parity of the weighted degrees and now $y_1$ and $y_2$ both have weighted degree 2. We redefine $w_{G'}$ accordingly, put $x$ back and give the edges $xy_1$ and $xy_2$ weight 0. This gives a proper \{0,1\}-weighting of $G$. \\
Now assume that $w_{G'}(y_1)=0$ and $w_{G'}(y_2) \geq 2$. Actually we can assume that $w_{G'}(y_2) = 2$ since otherwise if $w_{G'}(y_2) > 2$ we just repeat the proof of the previous case (after swapping the weights in $C$ the weighted degree of both $y_1$ and $y_2$ is at least 2). We can assume that all cycles going through $y_1$ in $G'$ also go through $y_2$ (otherwise we simply swap the weights in a cycle containing $y_1$ and not $y_2$). The only possible case is where $G-\{x,y_1,y_2\}$ consists of precisely two connected components $G_1, G_2$ with bipartition sets $X_i, Y_i$ each containing exactly one neighbour of both $y_1$ and $y_2$ (see Figure~\ref{fig:claim1}). Let $x_1,x_2$ denote the neighbours of $y_1$ in $G_1$ and $G_2$ respectively and let $z_1,z_2$ denote the neighbours of $y_2$ in $G_1$ and $G_2$ respectively. We allow the possibility that $x_1=z_1$ or $x_2=z_2$. If one of $X_1, X_2$ has even size, for example $X_1$, then the subgraph of $G_1$ consisting of all edges weighted 1 under $w_{G'}$ will have an odd number of odd degree vertices, which is not possible. So both  $X_1$ and  $X_2$ have odd size. The sets $Y_1$ and $Y_2$ must have different parity, in particular one of them, say, $Y_1$ has even size. Furthermore, if $G_2$ has a proper $\{0,1\}$-weighting $w_{G_2}$ then we can find a proper $\{0,1\}$-weighting of the whole graph $G$ with weight 0 on $y_1x_2$ and $y_2z_2$ as follows:  
If the weighted degrees of $x_2$ and $z_2$ have the same, say, odd parity under $w_{G_2}$, then because both bipartition sets in $G-G_2$ have even size, Lemma~\ref{lem:f-factor} implies that there is a proper \{0,1\}-weighing $w_{G-G_2}$ of $G-G_2$ where $y_1$ and $y_2$ get even weighted degree. We can now define a proper \{0,1\}-weighting $w_G$ of $G$ by $w_G(e)=w_{G_2}(e)$ for $e \in E(G_2)$, $w_G(e)=w_{G-G_2}(e)$ for $e \in E(G-G_2)$ and $w_G(y_1x_2)=w_G(y_2z_2)=0$. So the weighted degree of $x_2$ and $z_2$ do not have the same parity under $w_{G_2}$. Without loss of generality assume that $x_2$ has even weighted degree under $w_{G_2}$ and  $z_2$ has odd weighted degree under $w_{G_2}$. As before there is a proper \{0,1\}-weighting $w_{G-G_2}$ of $G-G_2$ where all vertices in $X_1 \cup \{x\}$ get odd weighted degree and all other vertices get even weighted degree. When extending $w_{G-G_2}$ and $w_{G_2}$ to the whole of $G$, the only possible conflict that can arise is $y_1x_2$, but we can always avoid this conflict by swapping the weights of the edges in a cycle in $G-G_2$ containing $y_1$. The same kind of argument shows that there is no proper $\{0,1\}$-weighting of $G_2$ where the weighted degree of both $x_2$ and $z_2$ are increased by 1 (increased by 2 if $x_2=z_2$) since we can let the edges $x_2y_1$ and $z_2y_2$ play the roles of the extra weights. By the minimality of $G$, the subgraph $G_2$ must either be an odd multi-cactus or contain a bridge. Lemma~\ref{lem:2} shows that $G_2$ cannot be an odd multi-cactus and hence it contains a bridge. Note that this shows that $x_2 \neq z_2$ and since $B$ is an endblock, one of $x_2, z_2$, say $x_2$, is not a cutvertex in $G_2$. If $x_1$ is also not a cutvertex in $G_1$ we do the following: Weight $G_2-x_2$ such that all vertices in $X_2 - \{x_2 \}$ have odd degree and all vertices in $Y_2$ have even degree, and weight $G_1-x_1$ such that all vertices in $X_1 - \{x_1 \}$ have odd degree and all vertices in $Y_1$ have even degree. These two $\{0,1\}$-weightings extend to the whole graph $G$ by assigning weight 0 to all edges incident to $x_1$, $x_2$ and $y_1$ except that we assign weight 1 to the edges $y_1x_1,$ and $y_1x_2$ and also to $y_1x$. So $x_1$ must be a cutvertex in $G_1$. Since $B$ is an end-block it follows that $z_2$ is not a cutvertex in $G_2$ and if $z_1$ is not a cutvertex in $G_1$ we do the same as before (with $x_1$ replaced by $z_1$ and $x_2$ replaced by $z_2$). The only possibility is that $x_1=z_1$ is a cutvertex. By Lemma~\ref{lem:f-factor} there is a $\{0,1\}$-weighting $w_{G_1}$ of $G_1$ where all vertices in $X_1\backslash \{x_1\}$ get odd weighted degree and all other vertices get even weighted degree, and a $\{0,1\}$-weighting $w_{G_2}$ of $G_2$ where all vertices in $X_2\backslash \{x_2\}$ get even weighted degree and all other vertices get odd weighted degree. We extend $w_{G_1}$ and $w_{G_2}$ to a $\{0,1\}$-weighting of the whole of $G$ by assigning weight 1 to the edges $y_1x_2$ and $y_2x_1$ and weight 0 to the edges $y_2x$, $y_1x$, $y_2z_2$ and $y_1x_1$. The only possible conflicts are between $y_1, y_2$ and $x_1=z_1$ if $x_1$ has weighted degree 1. If this is the case, then since $x_1$ is a cutvertex in $G_1$, there is a cycle in $G_1$ containing two edges with weight 0 incident to $x_1$. We can then swap the weights on such a cycle to avoid conflicts between $x_1$ and $y_1$ and $y_2$. This contradicts $G$ being a bad graph. 
\begin{figure}[H]
\centering 
\includegraphics[scale=0.8]{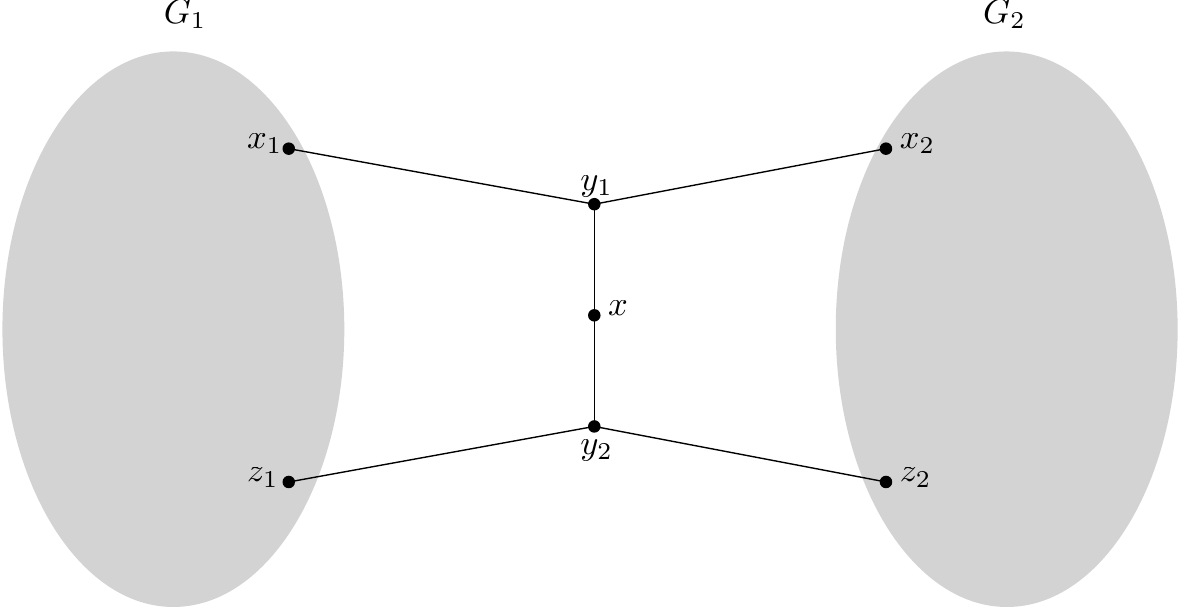}
\caption{An illustration of the situation in Claim 1.}
\label{fig:claim1}
\end{figure}

\noindent \textbf{Claim 2:} $B$ contains no suspended path or cycle of length 4. \\ \\
Assume that $y_1x_1y_2x_2y_1$ is a suspended cycle of length 4 in $B$. By Lemma~\ref{lem:f-factor} there is a proper $\{0,1\}$-weighting of $G-\{x_1,y_2,x_2\}$ where all vertices in $X \backslash \{x_1,x_2\}$ get even weighted degree and all vertices in $Y \backslash \{y_2\}$ get odd weighted degree. This proper $\{0,1\}$-weighting can now be extended to the whole graph by assigning weight 1 to the edges $y_1x_1$ and $x_2y_1$ and weight 0 to the edges $x_1y_2$ and $y_2x_2$, a contradiction. \\
The case where $y_1x_1y_2x_2y_3$ is a suspended path of length 4 in $G$ is treated in the same way as the suspended path of length 2 in the proof of the previous claim (here we just choose the graph $G-x_1-y_2-x_2$ as our $G'$). \\ \\
\textbf{Claim 3:} $G$ contains no suspended path or cycle of length at least 5. \\ \\
Suppose that $y_1x_1y_2x_2y_3x_3$ is a path in $G$ where $x_1 \in X$ and where the degree in $G$ of each of $x_1, y_2, x_2,y_3$ is 2. Now delete each of $x_1, y_2, x_2,y_3$ and add an edge $y_1x_3$ if there is not already such an edge. If the resulting graph is not an odd multi-cactus it has a proper $\{0,1\}$-weighting by the minimality of $G$. This $\{0,1\}$-weighting can now be used to find a proper \{0,1\}-weighting of $G$: we put back the vertices $x_1, y_2, x_2,y_3$. If $y_3x_3$ was not originally in $G$ we give $y_1x_1$ and $y_3x_3$ the same weight as $y_1x_3$ and delete that edge. We give $y_2x_2$ the opposite colour. Then we give $x_1y_2$ and $x_2y_3$ distinct colours. Since $y_1$ and $x_3$ have different colours, there are two choices for this and one of them will give a proper $\{0,1\}$-weighting. If $y_1x_3$ was in $G$ to begin with we assign weight 0 to the edges $y_1x_1$ and $y_3x_3$. We give $y_2x_2$ weight 1. Then we give $x_1y_2$ and $x_2y_3$ distinct colours. Again, there are two choices for this and one of them will give a proper $\{0,1\}$-weighting, a contradiction. So we can assume that $G'$ is an odd multi-cactus. Since $G$ is not an odd multi-cactus the only possibility is that $G$ is obtained from an odd multi-cactus by subdividing a green edge joining two vertices of degree at least 3 by four vertices. In this case we can find another path $y_1'x_1'y_2'x_2'y_3'x_3'$ where the degree in $G$ of each of $x_1', y_2', x_2',y_3'$ is 2 and define $G'$ from that such that $G'$ is not an odd multi-cactus, unless $G$ consists of two vertices joined by $s \geq 3$ paths of length 5. In this case it is easy to check that $G$ has a proper $\{0,1\}$-weighting. \\ \\
By Claims $1,2,3$ all degree-2 vertices in the endblocks of $G$ lie on a suspended path of length 3. In $G$ we replace all suspended paths of length 3 with an edge to form a multi-graph $G^*$. Edges arising from suspended paths will be called \emph{blue edges}. Note that $G^*$ is bridgeless and the minimum degree in any endblock is at least 3. Now let $B$ be an endblock of $G^*$. Possibly $G^*=B$. If $B \neq G^*$, then we let $x_0$ be the unique cutvertex of $G^*$ contained in $B$. If the deletion of some pair of neighbouring vertices in $B$ disconnects $G^*$, then we define a graph $B'$ as follows: we select an edge $y_0z_0$ in $B$ such that $G^*-y_0-z_0$ is disconnected and such that some component, $H$, not containing $x_0$ of $G^*-y_0-z_0$ is smallest possible. Possibly $x_0$ is one of $z_0$ and $y_0$. The union of that component, $H$, and $y_0, z_0$ together with all edges connecting them is called $B'$. Otherwise, if the deletion of any pair of adjacent vertices in $B$ leaves a connected graph we define $B'=H=B$. Note that in this case we must have $B'=H=B=G^*$, since the deletion of $x_0$ together with any of its neighbours disconnects $G^*$.  \\ \\
\textbf{Claim 4:} There is an end-block $B$ of $G^*$ such that all vertices in $H$ have degree 3 in $G^*$. \\ \\
The overall strategy for proving this claim is to find a vertex $w_0$ in $H$ with local maximum degree, and then use the procedure explained in the remark following Lemma~\ref{lem:remove-edges} to find a proper $\{0,1\}$-weighting of $G$ where all edges incident to $w_0$ have weight 1. \\ \\
Case 1: We can choose $B$ to be an end-block whose cutvertex is adjacent to only one other block. 
\\ \\
Suppose the claim is false and let $w_0 \in V(H)$ (if $B'=B$ choose $w_0$ distinct from $x_0$) be a vertex having maximum degree. We want to choose $w_0$ such that when we remove $w_0$ we avoid the exceptional case in Lemma~\ref{lem:remove-edges} (when $A$ is set to be the neighbours of $w_0$ having the same degree as $w_0$). If $d(w_0) \geq 5$ then this exceptional case cannot occur. If $d(w_0)=4$ and some neighbour of $w_0$ has degree 3 then the exceptional case is also avoided. Such a $w_0$ is possible to choose unless all vertices in $H$ have degree 4. If it is impossible to avoid the exceptional case then it must be that whenever we remove a vertex $w_0$ in $H$ and its four neighbours the resulting graph has six components each having exactly two neighbours in $N(w_0)$. If this is the case we choose $w_0$ to be such that the component arising when deleting $w_0$ and its neighbours (there are six components) containing $y_0$ or $z_0$ is maximal. The other components are easily seen to be isolated vertices (otherwise we redefine $w_0$ to be a neighbour of $w_0$ not joined to the component containing $y_0$ or $z_0$ and this will contradict the choice of $w_0$). But these isolated vertices must have degree 2, a contradiction. This shows that we can always find a $w_0$ of maximum degree in $H$ and avoid the exceptional case in Lemma~\ref{lem:remove-edges} (when $A$ is set to be the neighbours of $w_0$ having the same degree as $w_0$). We now choose such a $w_0$. When we have found and defined $w_0$ we go back to considering the original graph $G$. We will look at three different subcases: 
\begin{enumerate}
\item $w_0$ is not a neighbour of $z_0$ or $y_0$.
\item $w_0$ is a neighbour of $z_0$ and $z_0 \neq x_0$.
\item $w_0$ is a neighbour of $x_0$.
\end{enumerate} 
Subcase 1: This subcase is dealt with as described in the remark following Lemma~\ref{lem:remove-edges}. By the minimality of $H$, none of the neighbours of $w_0$ are incident to a bridge in $G-w_0$.   \\ \\
Subcase 2: We can assume that $z_0$ has degree strictly greater than that of $w_0$ since otherwise we do the same as in Subcase 1. This implies that the degree of $z_0$ is at least 5. We can assume that all vertices in $H$ having maximum degree are adjacent to $z_0$ or $y_0$, since otherwise we can redefine $w_0$ and go to Subcase 1. Note that this implies that we can never be in the exceptional case in Lemma~\ref{lem:remove-edges} when we delete a vertex $v$ in $H$ of maximum degree and define $A$ to be the neighbours of $v$ with the same degree as $v$.  \\ 
We can assume that $z_0$ has precisely one neighbour in each component other than $H$ in $B-y_0-z_0$, since otherwise we do the same as in Subcase 1 except we now also remove two edges from $z_0$ that go to the same component of $B-y_0-z_0$ other than $H$. If we then end up with a colour-conflict between $w_0$ and $z_0$ we have made sure that we can swap the weights in a cycle avoiding $H$ that contains two edges with weight 0 incident to $z_0$. This will then avoid the conflict between $w_0$ and $z_0$ and give a proper $\{0,1\}$-weighting of $G$. So $z_0$ has precisely one neighbour in each component other than $H$ in $B-y_0-z_0$. We can also assume that there is at most one component $C$ other than $H$ in $B-y_0-z_0$ which contains a neighbour of $z_0$ since otherwise we can remove two edges incident to $z_0$ going to two different components distinct from $H$ and use the same weight-swapping argument as before to avoid a conflict between $z_0$ and $w_0$ (this time the cycle will also go through $y_0$). \\
If $z_0$ has no neighbour in any component of $G-z_0-y_0$ other than $H$, then since $G-y_0-z_0$ is disconnected it must be the case that $y_0=x_0$. If this is the case we redefine $w_0$ to be $z_0$ and go to Subcase 3.  So we assume that there is such a component $C$ other than $H$ in $G-z_0-y_0$ containing a neighbour of $z_0$. We start by doing the same as before giving $w_0$ maximum weighted degree and assigning 0 to at least one edge $e_a$ incident to each neighbour $a$ of $w_0$ which has the same degree as $w_0$. We also assign weight 0 to the unique edge $z_0z_1$ joining  $z_0$ to the component $C$. Actually, since $w_0$ and all the neighbours of $w_0$ have the same weighted degree-parity, each of these neighbours of $w_0$ with the same degree as $w_0$ will be incident to at least two edges weighted 0. We can assume that $w_0$ and $z_0$ have the same weighted degree and the edge $z_0y_0$ has weight 1 (otherwise we can swap the weights in a cycle using the edges $z_0z_1$ and $z_0y_0$ to avoid the conflict between $w_0$ and $z_0$). If we swap the weights in a cycle in $G-w_0$ containing two edges incident to $z_0$ with the same weight, then the only conflicts we can create are between $w_0$ and a neighbour $v$ of $w_0$ with the same degree as $w_0$, and these conflicts can only arise when the cycle goes through the only two edges incident to $v$ with weight 0. If $v$ is a neighbour of $w_0$ with the same degree as $w_0$ and incident to only two edges with weight 0, then we call this pair of weight 0-edges a \textit{forbidden pair} of edges. \\ 
We will now show that we can always find a cycle in $G-w_0$ containing two edges incident to $z_0$ with the same weight that does not use any forbidden pair of edges. Note that all neighbours of $w_0$ which are incident with a forbidden pair of edges have the same degree as $w_0$ and are therefore neighbours of $y_0$. Let $v_1,...,v_m$ denote these neighbours of $w_0$. Since $z_0$ has weighted degree strictly greater than 3, there is a vertex $z_2 \neq w_0$ in $N(z_0) \cap V(H)$.
It suffices to find a path $P$ from $y_0$ to a vertex $z'$ in $N(z_0) \cap V(H)$ in the connected graph $G-z_0-w_0$ (connected by the minimality of $H$) not using any forbidden pair of edges, since then we can define our cycle to be $P \cup \{z_0y_0,z'z_0\}$ if the weight on $z_0z'$ is 1, or $P \cup P_c$, where $P_c$ is a path from $z_0$ to $y_0$ in $G-H-z_0y_0$ if the weight on $z_0z'$ is 0. See Figure \ref{fig:claim4}. Since the graph $G-z_0-w_0$ is connected, there is a path $P_1$ from $z_2$ to $y_0$. We can assume that this uses forbidden pairs of edges. Without loss of generality let $av_1$ and $v_1b$ be the first forbidden pair of edges $P_1$ uses when starting from $z_2$. Since $v_1$ is adjacent to $y_0$ it follows that $b=y_0$, since otherwise we can find a path from $y_0$ to $z_2$ not using any forbidden pair of edges. This shows that there is some path from $y_0$ to a vertex in $N(z_0) \cap V(H)$ only using one forbidden pair of edges. Now we look at all such paths only using one pair of forbidden edges $y_0v_i$ and $v_ia$ (for $i \in \{1,...,m\}$) and choose a path $P$ among those which goes through the most neighbours of $w_0$. Let $y_0v_i$ and $v_ia$ be the pair of forbidden edges $P$ contains. Since $y_0v_i$ and $v_ia$ is a forbidden pair, the vertex $v_i$ has a neighbour $v_i' \neq w_0$ in $H$ such that $v_iv_i'$ has weight 1. Since $G-w_0-v_i$ is connected it has a path $P'$ from $v_i'$ to a vertex in $N(z_0) \cap V(H)$. The path $P'$ must use a forbidden pair of edges, otherwise the graph induced by $E(P) \cup E(P')$ contains a desired path from $y_0$ to a vertex in $N(z_0) \cap V(H)$ avoiding forbidden pairs of edges. Let the first pair of forbidden edges $P'$ uses when starting from $v_i'$ be $bv$ and $vc$. The subpath $P_1'$ of $P'$ from $v_i'$ to $v$ must be disjoint from $P$, since otherwise the graph induced by $E(P) \cup E(P_1')$ contains a desired path from $y_0$ to $N(Z_0) \cap V(H)$ avoiding forbidden pairs of edges. Furthermore, we must have that $c=y_0$ since otherwise the path $P''$ defined to be $y_0v$ together with the subpath of $P'$ from $v$ to $v_i'$ followed by $v_i'v_i$ and the subpath of $P$ from $v_i$ to $N(z_0) \cap V(H)$ is a desired path from $y_0$ to $N(z_0) \cap V(H)$ avoiding forbidden pairs of edges. Now the path $P''$ contradicts the maximality of $P$. \\ 
This takes care of Subcase 2 if $z_0$ has a neighbour in some component $C$ other than $H$ in $G-z_0-y_0$. If this is not the case then, as noted above, we can go to Subcase 3 redefining $w_0$ to be $z_0$.
\begin{figure}[H]
\centering
\includegraphics[scale=1]{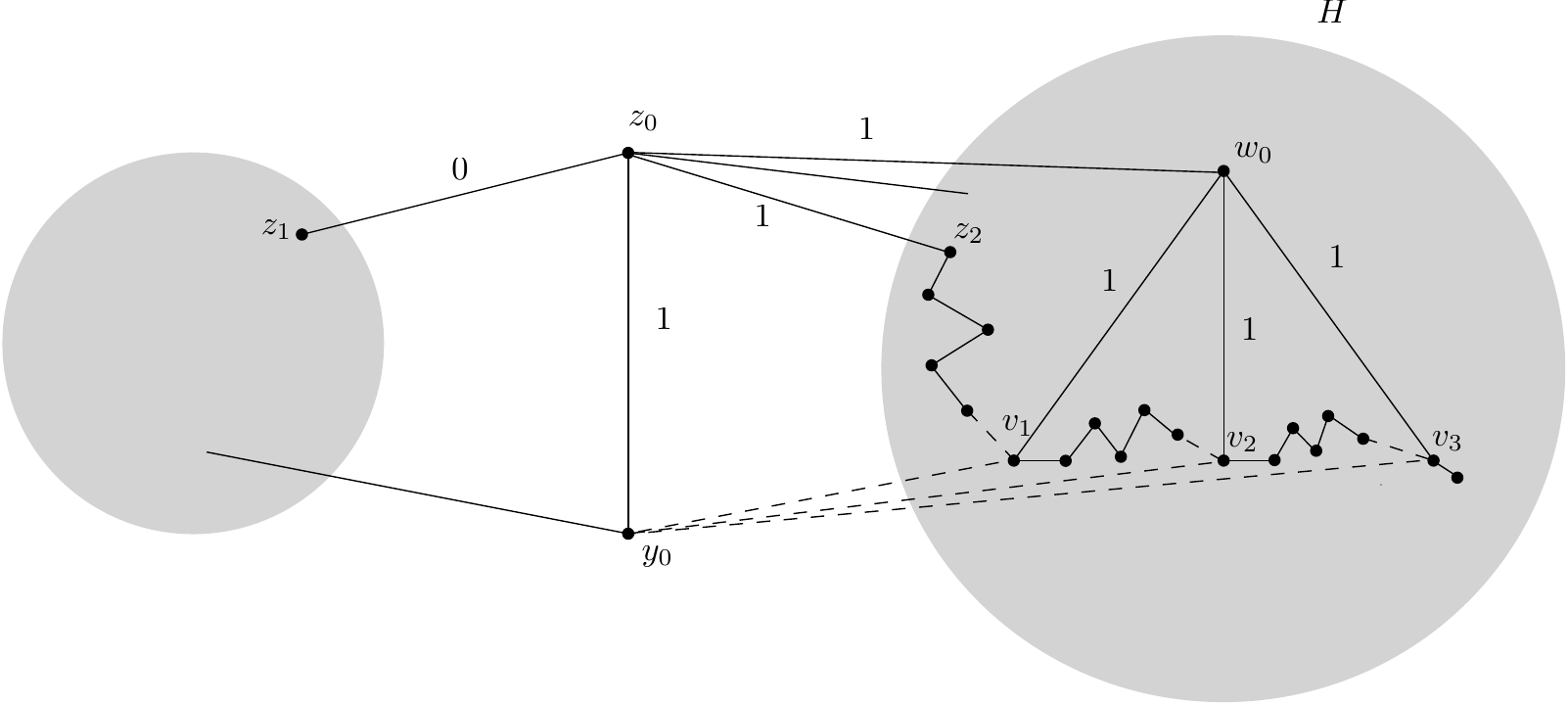}
\caption{An illustation of a situation in the proof of Claim 4. Dashed edges indicate pairs of forbidden edges.}
\label{fig:claim4}
\end{figure}
\noindent Subcase 3: The situation is more or less the same as in Subcase 2 except now $z_0=x_0$. If some vertex $w \in N(w_0)\backslash \{x_0\}$ has the same degree as $w_0$, then we can assume that we are in the exceptional case in Lemma~\ref{lem:remove-edges} when we remove $w$ and define $A$ to be the set of neighbours of $w$ with the same degree as $w$ (otherwise we redefine $w_0$ to be $w$ and go to Subcase 1 or 2). So in this case the degree of both $w_0$ and $w$ is 4 and so is the degree of all neighbours of $w$. Choose $w$ in $H$ to be a non-neighbour of $x_0$ with the same degree as $w_0$ such that the component arising when deleting $w$ and its neighbours (there are six components) containing $y_0$ or $z_0$ is maximal. The other components are easily seen to be isolated vertices, and this contradicts that the minimum degree in $H$ is 3. \\ 
So we can assume that $w_0$ is a strict local degree maximum in $V(B)\backslash \{x_0\}$ and $x_0$ has strictly greater degree than $w_0$ in $G$. This implies that $x_0$ has degree at least 5. Let $Y$ denote the bipartition set containing $w_0$ and let $X$ denote the opposite bipartition set. As before we find a $\{0,1\}$-weighting of $G$ where all edges incident to $w_0$ have weight 1, all vertices in $X$ have the same weighted degree parity as $w_0$, and all vertices in $Y\backslash \{w_0\}$ have weighted degree parity different from $w_0$. Now we can only have a conflict between $x_0$ and $w_0$. Recall that $x_0$ is only incident with two blocks $B$ and $B_1$. There must be precisely two neighbours $w_1$ and $w_2$ of $x_0$ in $B_1$, since otherwise we can avoid the conflict between $x_0$ and $w_0$ by swapping weights in a cycle in $B_1$ using two edges incident to $x_0$ with the same weight. By the same argument we can assume that the weights on the two edges $x_0w_1$ and $x_0w_2$ are different. Since $d(x_0) \geq 5$, this implies that $x_0$ must have at least two neighbours $w_3$ and $w_4$ in $B-w_0$ joined to $x_0$ by an edge weighted 1. The graph $B-x_0-w_0$ is connected by the minimality of $H$ so we can find a cycle in $B$ through the two edges $x_0w_3$ and $x_0w_4$ avoiding $w_0$. We swap the weights on this cycle and thereby avoid the conflict between $x_0$ and $w_0$.
\\ 
This completes Case 1. \\ \\
Since we can now assume that we are not in Case 1 we can go to Case 2 below by considering a longest path in the block-tree of $G^*$. \\ \\ 
Case 2: We can choose $B$ to be an end-block incident to endblocks $B_1,...,B_n$ where $n \geq 1$, and the union of all other blocks $B_{-1}$ satisfies that $B_{-1}-x_0$ is connected. \\ \\
In this case the proofs in Subcases 1 and 2 are exactly the same as before (the situation is now only different when $w_0$ is incident to $x_0$). For $i=-1,...,n$ define $G_i=B_i-x_0$. As before let $Y$ denote the bipartition set containing $w_0$ and let $X$ denote the opposite bipartition set. For $i=0,...,n$ let $w_i$ denote the vertex defined in the same way as $w_0$ just in $B_i$ instead of in $B$. As before we can assume that all neighbours of $w_i$ different than $x_0$ have strictly lower degree than that of $w_i$ and, furthermore, that $x_0$ has precisely two neighbours $v_{i,1}$ and $v_{i,2}$ in each $G_{-1},...,G_n$. We can assume that $w_i=v_{i,1}$ for $i=0,1,...,n$ and that the degree of $v_{i,2}$ is at most that of $w_i$, since otherwise we redefine $w_i$ to be $v_{i,2}$. For each $i=-1,...,n$, let $X_i$ and $Y_i$ denote the part of $G_i$ belonging to $X$ and $Y$ respectively. We can assume that we will get a conflict between $x_0$ and $w_i$ whenever we weight as before giving $w_i$ maximal weighted degree. As noted above, $x_0$ will get precisely weight 1 from each $G_j$ for $j \neq i$. So, for each $i=0,...,n$, the degree of $w_i$ is either $n+2$ or $n+3$.
We look at five different subcases:
\begin{enumerate}[label=(\alph*)]
\item $d(w_0)=n+2$ and $d(w_1)=n+3$ and $n$ is even.
\item $d(w_0)=n+2$ and $d(w_1)=n+3$ and $n$ is odd.
\item $d(w_i)=n+2$ for $i=0,1,...,n$.
\item $d(w_i)=n+3$ for $i=0,1,...,n$ and $n$ is even.
\item $d(w_i)=n+3$ for $i=0,1,...,n$ and $n$ is odd.
\end{enumerate}
(a): In this subcase $n$ is at least 2. Recall that when weighting $G$ as before giving $w_0$ weighted degree $n+2$ the vertex $x_0$ will have precisely one edge weighted 1 going to each $G_{-1},G_0,...,G_n$, and  when weighting $G$ as before giving $w_1$ weighted degree $n+3$, the vertex $x_0$ will get precisely weight 1 from all $G_{-1},G_0,G_2,G_3,...,G_n$ and weight 2 from $G_1$. The $\{0,1\}$-weighting giving $w_0$ maximum weighted degree implies that all the sets $Y_{-1}, Y_1, Y_2,...,Y_n$ have odd size and $Y_0$ has even size, since otherwise if $Y_i$ has even size for $i=-1,1,2,...,n$ or if $Y_0$ has odd size, then the subgraph of $G_i$ consisting of the edges with weight 1 has an odd number of vertices of odd degree. Similarly the $\{0,1\}$-weighting giving $w_1$ maximum weighted degree implies that all the sets $X_{-1},X_0,X_1,X_2,X_3,...,X_n$ have odd size. We find a proper $\{0,1\}$-weighting of $G$ as follows. For $i=-1,1,2,3,...,n$ we weight each $B_i$ by Lemma~\ref{lem:f-factor} such that all vertices in $X_i \cup {x_0}$ get odd weighted degree and all vertices in $Y_i$ get even weighted degree. We also find a $\{0,1\}$-weighting of $B_0$ such that all vertices in $Y_0$ get odd weighted degree and all vertices in $X_0 \cup \{x_0\}$ get even weighted degree. We can assume that $x_0$ gets weighted degree 2 (if the weighted degree of $x_0$ is 0 we swap the weights on a cycle containing the two edges incident to $x_0$). The union of these $\{0,1\}$-weightings gives a $\{0,1\}$-weighting of $G$ such that the only parity conflicts are between $x_0$ and its neighbours in $B_0$. However, the weighted degree of $x_0$ is $n+3$ while the neighbours of $x_0$ in $B_0$ have degree at most $n+2$.
\\ \\
(b): In this subcase $n$ is at least 3. By the same argument as in Subcase (a), all the sets $X_{-1}, X_1, X_2,...,X_n$ have odd size, $X_0$ has even size and all the sets $Y_{-1}, Y_0, Y_1, Y_2,...,Y_n$ have odd size. We find a proper $\{0,1\}$-weighting of $G$ as follows. For $i=-1,1,2,3,...,n$ we weight each $B_i$ by Lemma~\ref{lem:f-factor} such that all vertices in $X_i \cup {x_0}$ get odd weighted degree and all vertices in $Y_i$ get even weighted degree. We also find a $\{0,1\}$-weighting of $B_0$ such that all vertices in $Y_0 \cup \{x_0\}$ get even weighted degree and all vertices in $X_0$ get odd weighted degree. As in Subcase (a) we can assume that the weighted degree of $x_0$ is 2. The union of these $\{0,1\}$-weightings gives a proper $\{0,1\}$-weighting of $G$ (analogously to Subcase (a)). 
\\ \\
(c): First assume that $n$ is even. Then $n$ is at least 2. As before we deduce from the $\{0,1\}$-weighting of $G$ where $w_0$ gets weighted degree $n+2$ that all the sets $Y_{-1}, Y_1, Y_2,...,Y_n$ have odd size and $Y_0$ has even size. The same argument for the $\{0,1\}$-weighting of $w_1$ shows  that all the sets $Y_{-1}, Y_0, Y_2,...,Y_n$ have odd size and $Y_1$ has even size, a contradiction. An analogous argument holds when $n$ is odd.
\\ \\
(d): In this subcase $n$ is at least 2 and all the sets $X_{-1},X_0,X_1,X_2,X_3,...,X_n$ have odd size. We weight each $B_i $ for $i=0,1,2$ such that $w_i$ gets maximum weighted degree, $x_0$ gets weighted degree 2 and there are only parity conflicts around $w_0$. In all other blocks $B_j$, $j \neq i$, we weight such that all vertices in $X_j \cup \{x_0\}$ get odd weighted degree and all other vertices get even weighted degree. The union of these $\{0,1\}$-weightings gives a proper $\{0,1\}$-weighting of $G$. \\ \\
(e): In this subcase $n$ is at least 1 and all the sets $Y_{-1},Y_0,Y_1,Y_2,...,Y_n$ have odd size. One of the sets $X_{-1}, X_0,...,X_n$ must have even size. If $X_{-1}$ has even size we weight as follows: In $B_{-1}$ we weight such that all vertices in $X_{-1}$ get odd weighted degree and all vertices in $Y_{-1} \cup \{x_0\}$ get even weighted degree and, furthermore, such that $x_0$ has weighted degree 2. In $B_0$ we weight such that $w_0$ gets maximum weighted degree and all vertices in $X_0 \cup \{w_0, x_0\}$ get even weighted degree and all vertices in $Y_0 - \{w_0\}$ get odd weighted degree and, furthermore, such that the degree of $x_0$ is 2. In all other blocks $B_j$, $j \neq -1,0$ we weight such that all vertices in $Y_j \cup \{x_0\}$ get odd weighted degree and all other vertices get even weighted degree. The union of these $\{0,1\}$-weightings gives a proper $\{0,1\}$-weighting of $G$. \\
Hence we can assume that $X_{-1}$ has odd size. One of $X_0, X_1,..., X_n$, say, $X_0$ has even size and we now weight as follows: In $B_0$ we weight such that all vertices in $X_0$ get odd weighted degree and all vertices in $Y_0 \cup \{x_0\}$ get even weighted degree and, furthermore, such that $x_0$ has weighted degree 2. In $B_1$ we weight such that:
\begin{itemize}
\item $w_1$ gets maximum weighted degree.
\item All vertices in $X_1 \cup \{w_1, x_0\}$ get even weighted degree.
\item All vertices in $Y_1 - \{w_0\}$ get odd weighted degree.
\item The weighted degree of $x_0$ is 2.
\end{itemize}
In $B_{-1}$ we weight such that all vertices in $X_{-1} \cup \{x_0\}$ get odd weighted degree and all vertices in $Y_{-1}$ get even weighted degree. In all other blocks $B_j$, $j \notin \{-1,0,1\}$ we weight such that all vertices in $Y_j \cup \{x_0\}$ get odd weighted degree and all other vertices get even weighted degree. The union of these $\{0,1\}$-weightings gives a proper $\{0,1\}$-weighting of $G$.
\\ \\
This completes the proof of Claim 4. \\ \\
If the removal of any pair of adjacent vertices leaves a connected graph we must have that $G^*$ is 3-regular and we will simply work in $G^*$ from now on. Otherwise we choose to work in an endblock $B$ of $G^*$ and the subgraph $H$ of $B$ defined before Claim 4. By Claim 4, all vertices of $H$ have degree 3. Suppose first that all vertices in $H$ are adjacent to $z_0$ or $y_0$. A small argument shows that unless $B$ is isomorphic to $K_{3,3}$, there is a vertex $w_0 \in V(H)$, such that removing $w_0$ and all the neighbours $w_0$ would leave a connected graph. In this case we can find a proper $\{0,1\}$-weighting of $G$ by Lemma~\ref{lem:f-factor}. If $B$ is isomorphic to $K_{3,3}$ we remove all vertices in $B$ except $x_0$. The resulting subgraph of $G$ has an odd number of vertices so by Lemma~\ref{lem:f-factor} it has a proper $\{0,1\}$-weighting without parity-conflicts. Some edges in $B$ may be blue, but it can be checked that no matter how these blue edges are arranged in $B$ this $\{0,1\}$-weighting can be extended to the whole of $G$. So we can assume that there is some vertex in $H$ not adjacent to $x_0$ or $z_0$.    \\ \\
The rest of the proof is as that of Theorem~\ref{thm:12oddmul} in~\cite{TWZ} (choose $w_0$ to be a non-neighbour of $z_0$ and $y_0$, in $H$). This completes the proof of the theorem. 
\end{proof}

\section{Trees without the \{0,1\}-property} \label{sec2}
In this section we will give a complete characterisation of all bad trees. The characterisation consists of a recursive construction using three other classes of trees with certain properties, and immediately gives a polynomial-time algorithm for recognising bad trees. We begin by defining these properties for general bipartite graphs. The first of these three classes is described as follows. Let $v$ be a vertex in a connected bipartite graph $G$ with an even number of vertices in each bipartition set. We say that $G$ is a $G_v(-)$-graph if there is no proper \{0,1\}-weighting of $G$ when the weighted degree of $v$ is increased by 1. This definition is motivated by the following easy proposition.
\begin{proposition} \label{prop:-tree}
Let $G$ be a graph and let $v$ be a vertex in $G$. Let $G'$ be the graph obtained from $G$ by adding two vertices $v_1$ and $v_2$ and the the edges $vv_1$ and $v_1v_2$. The graph $G$ is a $G_v(-)$-graph if and only if $G'$ is bad.  
\end{proposition}
\noindent The following two lemmas show a recursive way to construct new bad bipartite graphs from other bad bipartite graphs with vertices of degree 1. These two results hold for all bipartite graphs and not just for trees.
\begin{lemma} \label{lem: degree1vbridge}
Let $G$ be a simple connected bipartite graph without the $\{0,1\}$-property. If $v$ is a vertex of degree 1, and $v'$ is the unique neighbour of $v$, then all edges incident to $v'$ are bridges in $G$. 
\end{lemma}
\begin{proof}
By Lemma~\ref{lem:f-factor} there is a $\{0,1\}$-weighting of $G-v$ with no parity conflicts. The only problem we can have in extending this $\{0,1\}$-weighting to $G$ is that the weighted degree of $v'$ might be 0. If $v'$ is contained in a cycle we would always be able to avoid this.
\end{proof}
\begin{lemma} \label{lem:degree1allbad}
Let $G$ be a simple connected bipartite graph and assume that $v$ is a vertex of degree 1. Let $v'$ denote the neighbour of $v$ and let $e_0,e_1,...,e_n$ be the edges incident to $v'$ where $e_0=vv'$. Assume that all edges incident to $v'$ are bridges and for each $i>0$, let $G_i$ be the unique component of $G-e_i$ not containing $v$. For each $i>0$, let $G_i'$ denote the connected graph obtained from $G_i$ by adding the vertices $v, v'$ and the edges $e_0,e_i$. The graph $G$ is bad if and only if all the graphs $G_1',...,G_n'$ are bad.
\end{lemma}
\begin{proof}
Figure~\ref{fig:1degreev} shows an illustration of the situation. For $i \in \{1,...,n\}$, let $v_i$ be the vertex of $G_i$ which is adjacent to $v'$ in $G_i'$. If all $G_1',...,G_n'$ are bad then by Proposition~\ref{prop:-tree} each $G_i$ is a $G_{v_i}(-)$-graph. It follows that in any proper $\{0,1\}$-weighting of $G$, each edge $e_i$ must receive weight 0. But now $v$ and $v'$ have the same weighted degree. Thus no such $\{0,1\}$-weighting of $G$ exists, that is, $G$ is bad. \\
Now assume that $G$ is bad. Let $X, Y$ denote the bipartition sets of $G$ such that $v \in X, v' \in Y$ and for each $i=1,...,n$,  let $X_i, Y_i$ denote the bipartition sets of $G_i$. By Lemma~\ref{lem:f-factor} we can assume that both $X$ and $Y$ have odd size.  By Lemma~\ref{lem:f-factor} there is a  $\{0,1\}$-weighting of $G-v$ with no parity conflicts, where all vertices in $X-v$ get odd degree and all vertices in $Y$ get even degree. The only problem we can have in extending this $\{0,1\}$-weighting to $G$ is that the weight of $v'$ can be 0. If this is the case then all $X_i$ have even size. There must be an even number $m$ of the sets $Y_i$ which have an odd number of vertices. If $m \geq 2$, say $Y_1,...,Y_m$ have even size, then by Lemma~\ref{lem:f-factor} there is a proper $\{0,1\}$-weighting of $G$ where $v'$ gets weighted degree $m+1$ (apply Lemma \ref{lem:f-factor} to $G-v$ to find a $\{0,1\}$-weighting of $G-v$ where all vertices in $Y \setminus \{v'\}$ get odd weighted degree and all vertices in $(X-\{v\}) \cup \{v'\}$ get even weighted degree. In such a weighting the weights on all the edges $e_1,...,e_m$ are 1 and the weights on the other $e_i$'s are zero. Now extend this weighting to the whole of $G$ by assigning weight 1 to $e_0=vv'$.) This contradicts that $G$ is bad. So all $Y_i$'s have even size. 
By Proposition \ref{prop:-tree} each $G_i'$ is bad if and only if $G_i$ is bad when the weight on the vertex incident to $e_i$ is increased by 1. So for a contradiction assume that there is a proper $\{0,1\}$-weighting of some $G_i$ when the weight on the vertex incident to $e_i$ is increased by 1. By use of Lemma \ref{lem:f-factor} this proper $\{0,1\}$-weighting can now easily be extended to $G$, a contradiction.
\end{proof}
\begin{figure} [H]
\centering
\includegraphics[scale=1]{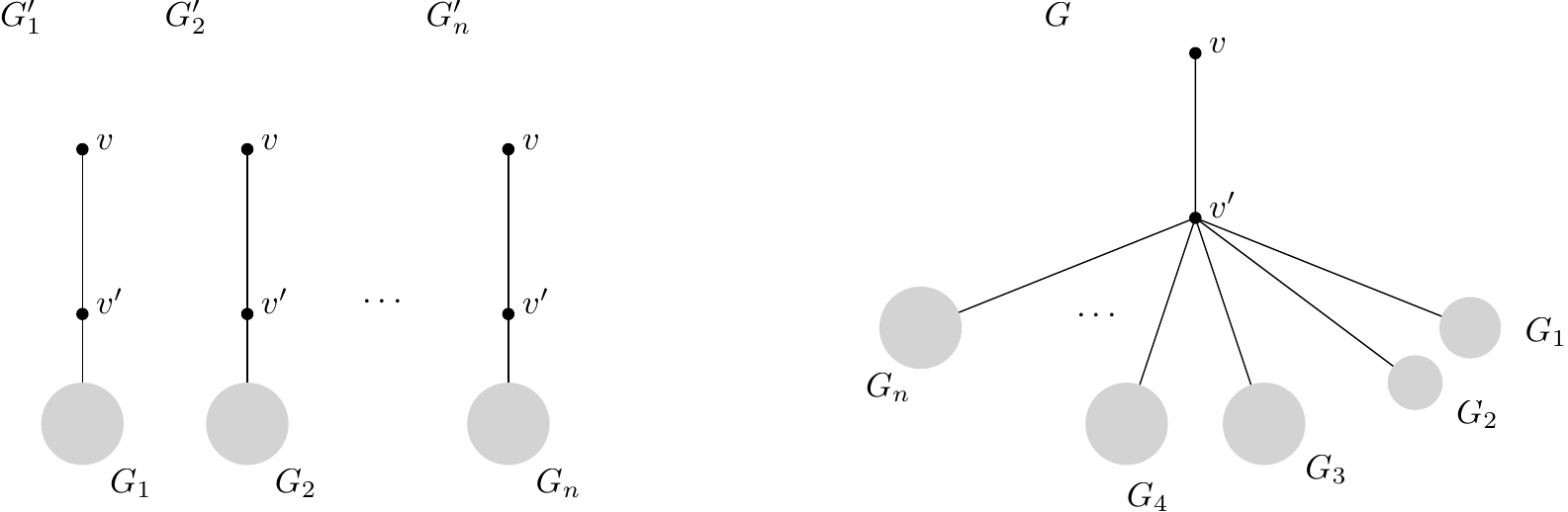}
\caption{An illustration of the situation explained in Lemma~\ref{lem:degree1allbad}.}
\label{fig:1degreev}
\end{figure}
\noindent We now describe the second and third class of trees we will use to characterise all bad trees. They are special cases of the graphs defined as follows. \\ 
Let $v$ be a vertex in a connected bipartite graph $G$ with an odd number of vertices and let $a,b$ be two non-negative integers. We say that $G$ is a $G_v(a,b)$\textit{-graph} if $v$ must get weighted degree $a$ in all proper $\{0,1\}$-weightings of $G$ and $v$ must get weighted degree $b$ in all proper $\{0,1\}$-weightings of $G$ where the weight of $v$ is increased by 1. \\
The classes of $G_v(s,s+1)$- and $G_v(s,s+3)$-trees where $s$ is a non-negative integer are two interesting special cases when we want to characterise all bad trees. We will need the following two lemmas describing the local structure around $v$ in a $G_v(s,s+1)$- and a $G_v(s,s+3)$-tree.
\begin{lemma} \label{lem:x,x+3tree}
Let $s$ be a non-negative integer and let $G$ be a $G_v(s,s+3)$-tree. Then $G$ is obtained from the disjoint union of a $G_{v_1}(s+1,s+2)$- and a $G_{v_2}(s+1,s+2)$-tree together with a number of trees of type $G_{v_3}(-), G_{v_4}(-),...,G_{v_m}(-) $ and bad trees $G_{m+1},...,G_{m+s}$  by adding a vertex $v$ and all edges $vv_1, vv_2, ..., vv_m$ and also an edge from $v$ to all the bad trees $G_{m+1},...,G_{m+s}$.
\end{lemma}
\begin{proof} 
Figure~\ref{fig:a3atreelemma} illustrates the situation. Assume that $s$ is even (the case where $s$ is odd is similar). Let $X,Y$ be the bipartition sets of $G$ where $v \in Y$. Let $e_1,...,e_n$ denote the edges incident to $v$ and let $G_1,...,G_n$ denote the corresponding components of $G-v$. Let $X_i, Y_i$ denote the bipartition sets of $G_i$. Let $s'$ denote the number of trees with an odd number of vertices in both bipartition sets and let $m'$ denote the number trees with an even number of vertices in both bipartition sets among $G_1,...,G_n$. Let $n_1$ denote the number of trees among $G_1,...,G_n$ that have an even number of vertices in their $X$-bipartition and an odd number of vertices in their $Y$-bipartition, and assume that the ordering of $G_1,...,G_n$ is such that $G_1, ..., G_{n_1}$ denote these trees. Let $n_2$ be the number of trees among $G_1,...,G_n$ that have an odd number of vertices in their $X$-bipartition and an even number of vertices in their $Y$-bipartition and assume that the ordering of $G_1,...,G_n$ is such that $G_{n_1+1}, ..., G_{n_1+n_2}$ denote these trees. Furthermore assume that the trees $G_{n_1+n_2+1},...,G_{n_1+n_2+s'}$ have an odd number of vertices in both bipartition sets and that the trees $G_{n_1+n_2+s'+1},...,G_{n_1+n_2+s'+m'}=G_n$ have an even number of vertices in both bipartition sets. For $i=1,...,n$ let $v_i$ be the neighbour of $v$ in $G_i$. 
\\
Since $|V(G)|$ is odd, one of $|X|$, $|Y|$ is even. However, if $|Y|$ is even, then by Lemma~\ref{lem:f-factor}, $G$ has a proper $\{0,1\}$-weighting such that $v$ gets odd weighted degree. This contradicts $G$ being a $G_v(s,s+3)$-tree. Thus $|X|$ is even and $|Y|$ is odd, and $G$ has a $\{0,1\}$-weighting such that all vertices in $Y$ get even weighted degree and all vertices in $X$ get odd weighted degree. In such a $\{0,1\}$-weighting all the edges $vv_1,..., vv_{n_1}$ must get weight 0, since otherwise if say $vv_1$ is weighted 1, then the subgraph consisting of edges weighted 1 in $G_1$ has an odd number of odd degree vertices. By a similar argument, all the edges $vv_{n_1+1},...,vv_{n_1+n_2}$ get weight 1, all the edges $vv_{n_1+n_2+1},...,vv_{n_1+n_2+s'}$ also get weight 1 and all the edges $vv_{n_1+n_2+s'},...,vv_n$ get weight 0. It follows that $n_2+s'=s$. By Lemma~\ref{lem:f-factor}, there is also a $\{0,1\}$-weighting of $G$ where all vertices in $Y\backslash \{v\}$ get odd weighted degree and all vertices in $X$ get even weighted degree. This means that there is a $\{0,1\}$-weighting of $G$ where all vertices in $Y$ get odd weighted degree and all vertices in $X$ get even weighted degree when the weight on $v$ is increased by 1. We argue as before and see that in such a $\{0,1\}$-weighting all the edges $vv_1,..., vv_{n_1}$ must get weight 1, all the edges $vv_{n_1+1},...,vv_{n_1+n_2}$ get weight 0, all the edges $vv_{n_1+n_2+1},...,vv_{n_1+n_2+s'}$ get weight 1 and all the edges $vv_{n_1+n_2+s'},...,vv_n$ get weight 0. Since $G$ is a $G_v(s,s+3)$-tree it follows that $n_1+s'+1=s+3$ and hence $n_1=n_2+2$.
\\ 
We start by showing that all the trees $G_{n_1+n_2+s'+1},...,G_n$ must be trees of type $G_{v_j}(-)$. Assume that this is not the case and let $G_j$ be a tree among $G_{n_1+n_2+s'+1},...,G_n$ such that there is a $\{0,1\}$-weighting of $G[V(G_j) \cup \{v\}]$ where the weight on $vv_j$ is 1 and the only possible conflict is between $v$ and $v_j$. Now we weight $G-G_j$ as before such that all vertices in $Y-Y_j$ get odd weighted degree and all vertices in $X-X_j$ get even weighted degree when the weight on $v$ is increased by 1. We now put back $G_j$ and let $vv_j$ play the role of the extra weight on $v$ which then has weight $s+3$. The only possible conflict is between $v$ and $v_j$, and since $G$ is a $G_v(s,s+3)$-tree we must have a conflict, so $v_j$ will also get weighted degree $s+3$. Now we weight $G-G_j$ such that all vertices in $Y-Y_j$ get even weighted degree and all vertices in $X-X_j$ get odd weighted degree. Now we put back $G_j$ and let $vv_j$ play the role of an extra weight on $v$ and we also increase the weight 1 on $v$. The weight on $v$ is then $s+2$ and we have no conflicts anywhere, a contradiction. So all the trees $G_{n_1+n_2+s'+1},...,G_n$ must be trees of type $G_{v_j}(-)$. Similar arguments show all the trees $G_{n_1+n_2+1},...,G_{n_1+n_2+s'}$ must be bad trees. \\ 
It remains to show that $n_1=2$ and $n_2=0$, and that the two graphs $G_1$ and $G_2$ are trees of type $G_{v_1}(s+1,s+2)$ and $G_{v_2}(s+1,s+2)$. We start by showing that $n_1=2$ and $n_2=0$. Clearly $n_1 \geq 2$. First assume that $n_1 >2$ is even. For any $j=1,...,n_1$ there is a $\{0,1\}$-weighting of $G$ where the weight on $vv_j$ is 1, the weight on all $vv_i$ for $i \neq j$ and $i \in \{1,...,n_1\}$ is 0 and the weight on all edges $vv_{n_1+1},..., vv_{2n_1-1}$ is 1 and the only possible conflict when the weight on $v$ is increased by 1 is $vv_j$.  So for each $G_i$, $i=1,...,n_1$ the weight of $v_i$ must be $s'+n_2+2=s'+n_1$ when the weight on $v_i$ is increased by 1, otherwise there is a proper $\{0,1\}$-weighting of $G$ where the weight on $v$ is increased by 1 up to $s'+n_1$. But now there is a proper $\{0,1\}$-weighting of $G$ with weight 1 on all edges $vv_1, ..., vv_{n_1+n_2}$ such that $v$ gets weighted degree $s'+n_1+n_2$, and this contradicts $G$ being a $G_v(s,s+3)$-tree. The case where $n_1 > 2$ and $n_1$ is odd is similar. \\
We conclude that $n_1=2$ and $n_2=0$ and it remains to show that $G_1$ and $G_2$ are trees of type $G_{v_1}(s+1,s+2)$ and $G_{v_{2}}(s+1,s+2)$. By Lemma~\ref{lem:f-factor} there is a $\{0,1\}$-weighting of $G$ such that the weight on the edges $vv_2,...,vv_{2+s'}$ is 1 and the weight on the other edges incident to $v$ is 0, and where the only possible conflict is between $v$ and $v_1$. This must be a conflict since $G$ is a $G_v(s,s+3)$-tree. So the weighted degree of $v_1$ in any proper $\{0,1\}$-weighting of $G_1$ must be $s+1$. If we use the same $\{0,1\}$-weighting, except we now swap the weighted degree parities in the trees $G_3,...,G_n$  and increase the weighted degree of $v$ by 1 we can similarly conclude that the weighted degree of $v_2$ in any proper $\{0,1\}$-weighting of $G_2$ where the degree of $v_2$ is increased by 1 must be $s+2$. Interchanging $G_1$ and $G_2$ in the argument above implies that the weighted degree of $v_2$ in any proper $\{0,1\}$-weighting of $G_2$ must be $s+1$ and that the weighted degree of $v_1$ in any proper $\{0,1\}$-weighting of $G_1$ where the degree of $v_1$ is increased by 1 must be $s+2$. Hence $G_1$ and $G_2$ are trees of type $G_{v_1}(s+1,s+2)$ and $G_{v_{2}}(s+1,s+2)$.
\end{proof}
\begin{figure} [H]
\centering
\includegraphics[scale=1]{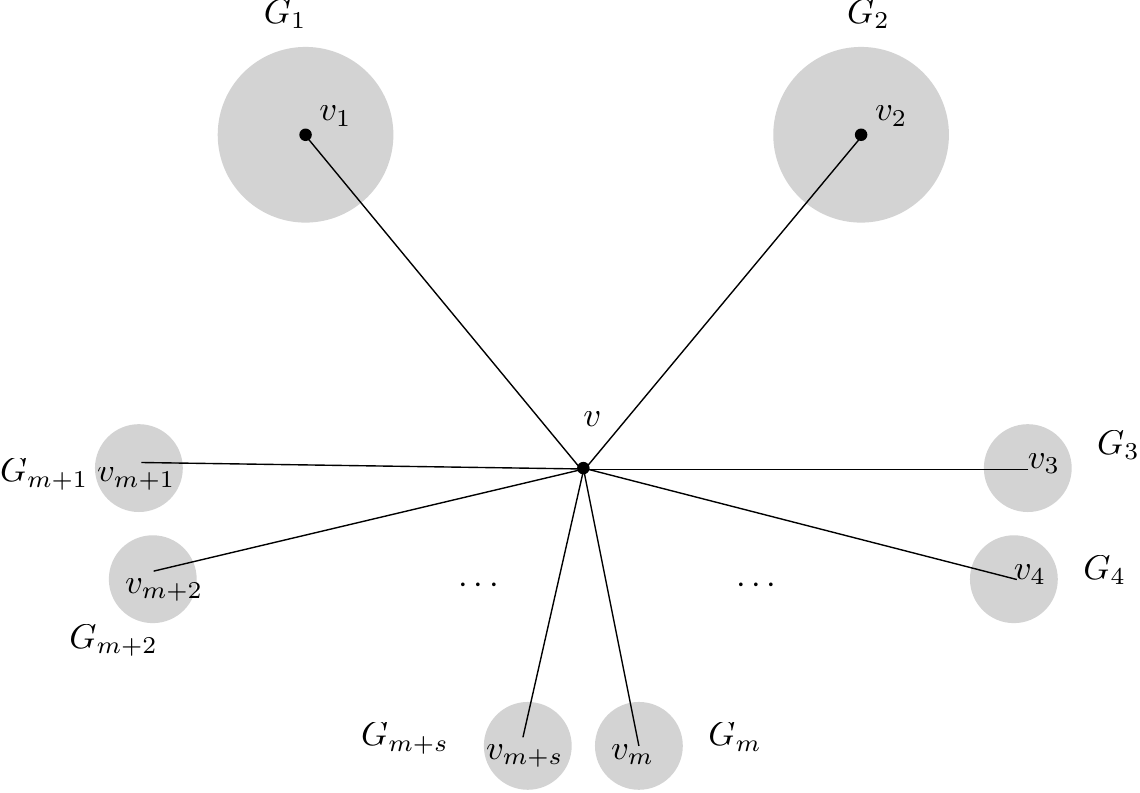}
\caption{An illustration of the situation explained in Lemma~\ref{lem:x,x+3tree}.}
\label{fig:a3atreelemma}
\end{figure}
\noindent Similarly to what we did in the proof of Lemma \ref{lem:x,x+3tree} we can describe the local structure around $v$ in a $G_v(s,s+1)$-tree.
\begin{lemma}\label{lem:x,x+1-tree}
If $G$ is a $G_v(s,s+1)$-tree, then either
\begin{enumerate}[label=(\alph*)]
\item $G$ is obtained from the disjoint union of a $G_{v_1}(s-1,s+2)$- and a $G_{v_2}(s,s+1)$-tree together with a number of trees of type $G_{v_3}(-), G_{v_4}(-),...,G_{v_m}(-)$ and bad trees $G_{m+1},...,G_{m+s-1}$ by adding a vertex $v$ and all edges $vv_1, vv_2, ..., vv_m$ and also an edge from $v$ to all the bad trees $G_{m+1},...,G_{m+s-1}$, or
\item  $G$ is obtained from the disjoint union of $s$ bad graphs $B_1, ..., B_s$ and a number of graphs of type $G_{v_1}(-), G_{v_2}(-),...,G_{v_n}(-)$  by adding a vertex $v$ and all edges $vv_1, vv_2, ..., vv_n$ and also bridges joining $v$ to each of the bad graphs.
\end{enumerate}
\end{lemma}
\begin{figure} [H]
\centering
\includegraphics[scale=0.8]{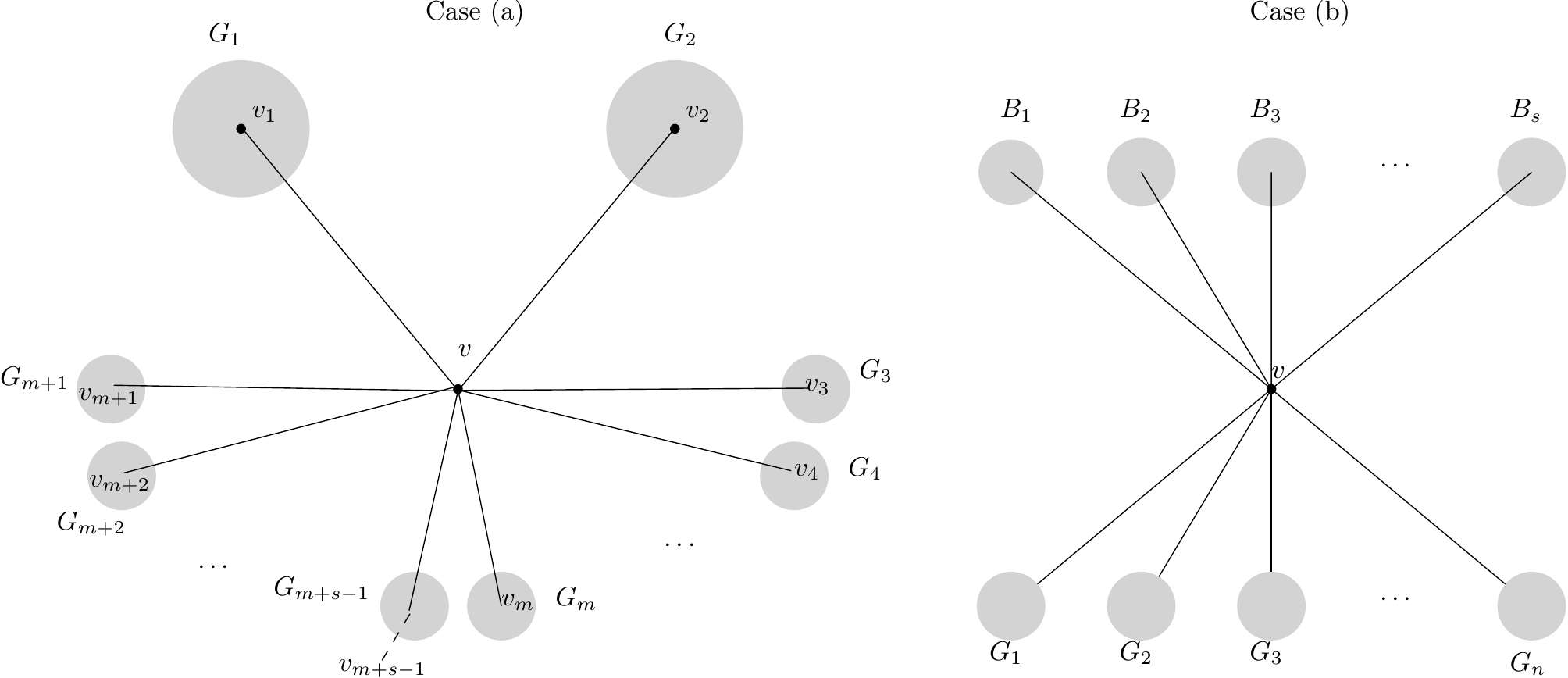}
\caption{An illustration of the two possible situations explained in Lemma~\ref{lem:x,x+1-tree}.}
\label{fig:a2atreelemma}
\end{figure}
\begin{lemma} \label{lem:allbadtrees}
Any bad tree distinct from $K_2$ is obtained from either: 
\begin{enumerate}[label=(\alph*)]
\item a $G_v(s,s+1)$-tree where $s>1$ by adding a vertex $v'$ joined to $v$ by an edge and to $s$ $K_2$-graphs by bridges, or
\item from two bad trees $G_1$ and $G_2$ by gluing together two edges $v_1'v_2'$ and $v_1''v_2''$ in $G_1$ and $G_2$ respectively where both $v_1'$ and $v_1''$ have degree 1 in $G_1$ and $G_2$ respectively and both $v_2'$ and $v_2''$ have degree $2$ in $G_1$ and $G_2$ respectively. 
\end{enumerate}
\end{lemma}
\begin{proof}
Suppose the lemma is false and look at a smallest counterexample $G$. It is easy to check that the statement holds for all bad trees of diameter at most 3. So we can assume that the diameter of $G$ is at least 4 and by Lemma \ref{lem:degree1allbad} we can also assume that all vertices of degree 1 are adjacent to vertices of degree 2. \\
We let $v$ be the fourth last vertex in a longest path in $G$ and let $v'$ be the third last vertex. Then the two subtrees obtained by removing the edge $vv'$ form the desired construction of $G$. 
\end{proof}

\noindent We list a recursive way to construct bad trees below in Figures~\ref{fig:consa3atrees},~\ref{fig:consminustrees},~\ref{fig:consa2atrees} and~\ref{fig:consbadtrees}. The class of bad trees which can be obtained in this way starting with $K_2$ as the smallest bad graph is denoted $\mathcal{B}$.

\begin{figure} [H]
\centering
\includegraphics[scale=0.8]{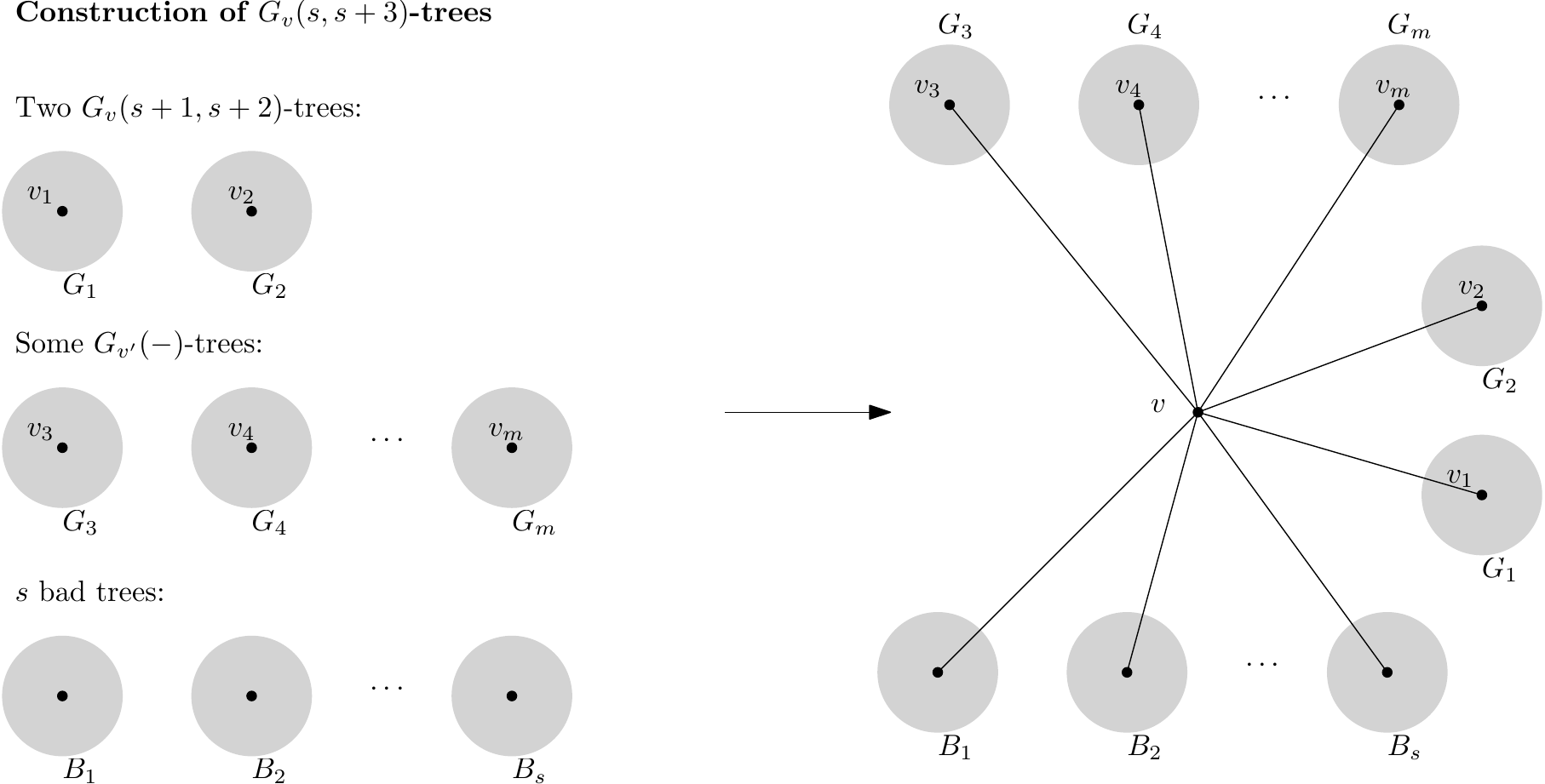}
\caption{This illustrates a recursive way how to construct all $G_v(s,s+3)$-trees.}
\label{fig:consa3atrees}
\end{figure}
\begin{figure} [H]
\centering
\includegraphics[scale=0.8]{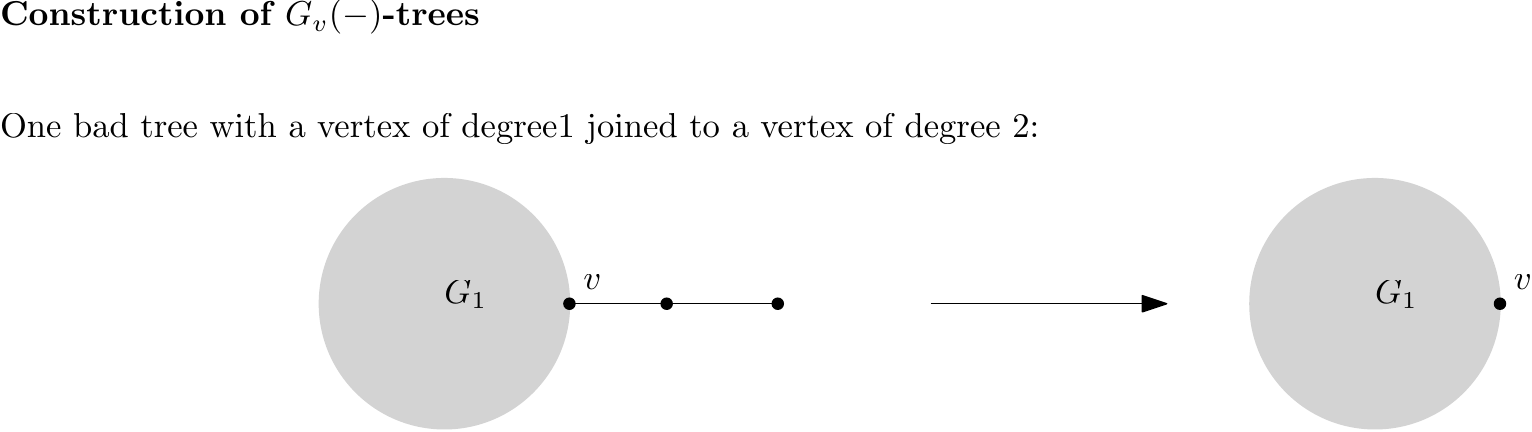}
\caption{This illustrates how to obtain all $G_v(-)$-trees from bad trees with 2 more vertices.}
\label{fig:consminustrees}
\end{figure}
\begin{figure} [H]
\centering
\includegraphics[scale=0.8]{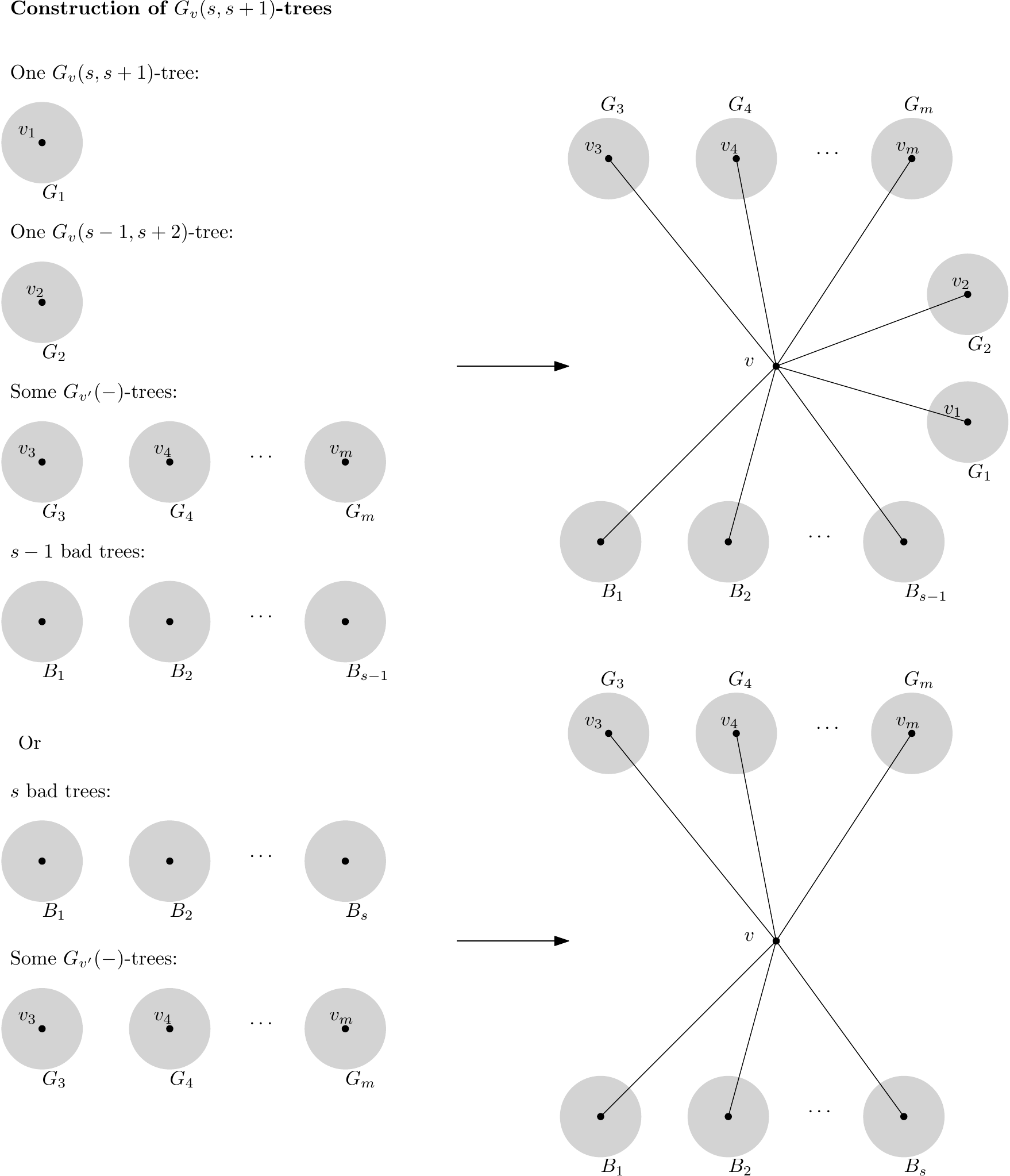}
\caption{This illustrates a recursive way how to construct all $G_v(s,s+1)$-trees.}
\label{fig:consa2atrees}
\end{figure}
\begin{figure} [H]
\centering
\includegraphics[scale=0.8]{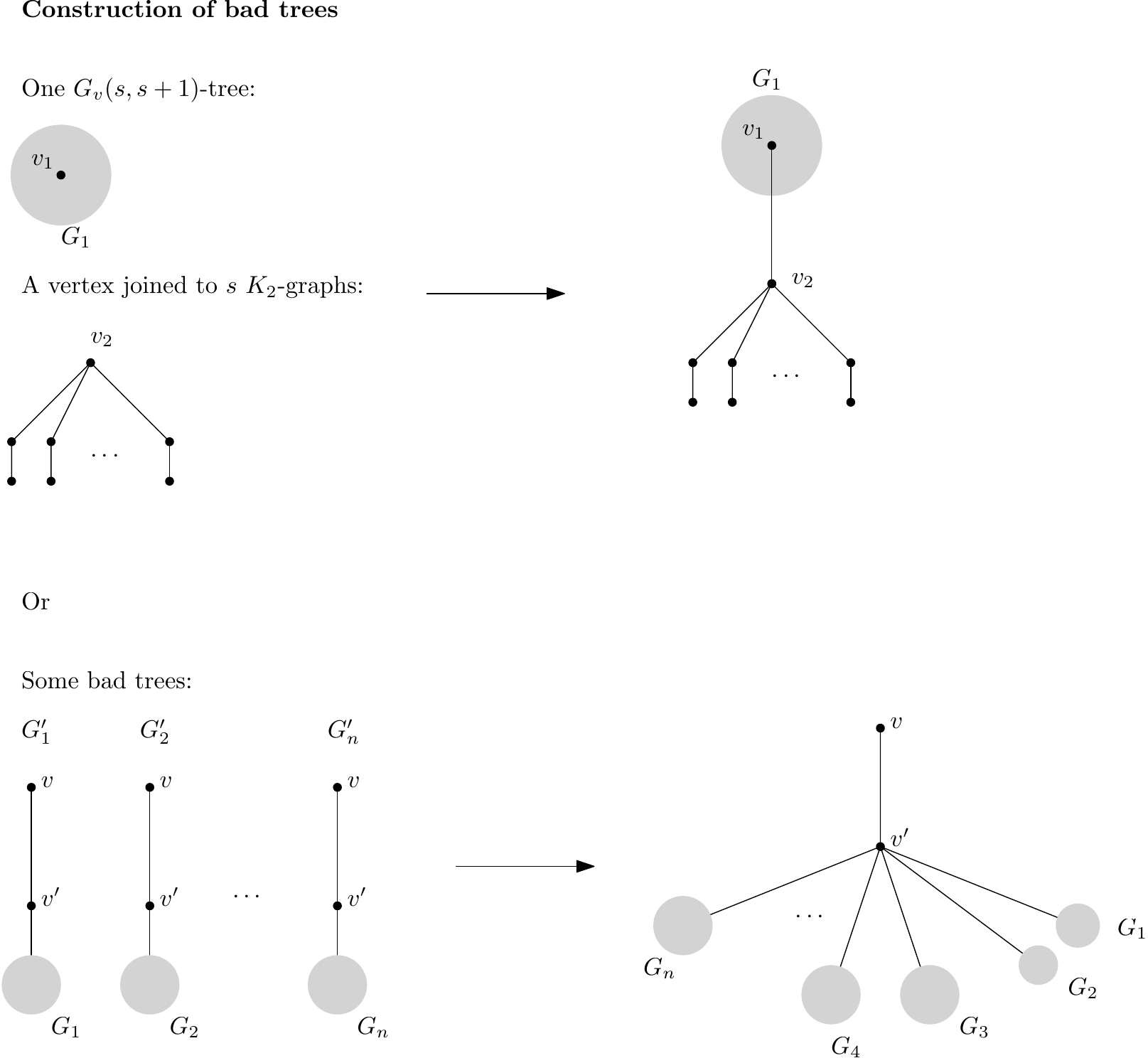}
\caption{This illustrates a recursive way how to construct all bad trees.}
\label{fig:consbadtrees}
\end{figure}
\noindent The above constructions do indeed describe all bad trees:

\begin{proof} [of Theorem~\ref{Thm:2}]
Suppose the theorem is false and let $G$ be a smallest bad tree which cannot be constructed by the above recursion.
It is easy to check that the diameter of $G$ must be at least 4. Let $n$ be the number of vertices in $G$. By Proposition~\ref{prop:-tree} and Lemmas~\ref{lem:x,x+3tree} and~\ref{lem:x,x+1-tree} we can assume that all trees of type $G_v(-)$ with at at most $n-4$ vertices and all trees of type $G_v(s,s+1)$ and $G_v(s,s+3)$ with at most $n-3$ vertices can be constructed using the above recursion. Furthermore, since $G$ is a smallest counterexample all bad trees with fewer vertices can also be constructed by the recursion. Lemma \ref{lem:degree1allbad} implies that $G$ cannot have a vertex of degree 1 which is adjacent to a vertex of degree at least 3. So by Lemma \ref{lem:allbadtrees} our counterexample $G$ is obtained from a $G_v(s,s+1)$-tree $G'$ where $s>0$ and a vertex joined to $s$ $K_2$-graphs by bridges. But $G'$ has at most $n-3$ vertices so $G'$ can be constructed by the recursion, and then so can $G$, a contradiction.
\end{proof}
\section{Concluding remarks}
We have provided a characterisation of all bridgeless bipartite graphs without the \{0,1\}-property and all trees without the \{0,1\}-property. Actually, since the \{0,1\}-property is equivalent to the $\{0,a\}$-property for any non-zero integer $a$ these characterisations extend to the $\{0,a\}$-property. The characterisations also provide polynomial time algorithms checking the $\{0,a\}$-property. This, together with Theorem \ref{thm:12oddmul} from \cite{TWZ}, answers Problem 1 in \cite{Lu} except for bipartite graphs with bridges. So it remains to characterise all the bipartite graphs with bridges and without the \{0,1\}-property. It would be interesting to investigate whether the methods used in Section 3 can be extended to characterise all bipartite graphs without the \{0,1\}-property.
\acknowledgements
\label{sec:ack}
The author would like to thank Carsten Thomassen for advice and helpful discussions, as well as Thomas Perret for careful reading of the manuscript.

\nocite{*}

\bibliographystyle{abbrvnat}

% use the following instead if you encounter problems 
%\bibliographystyle{alpha}
\bibliography{bibilibob}

\begin{thebibliography}{7}
\providecommand{\natexlab}[1]{#1}
\providecommand{\url}[1]{\texttt{#1}}
\expandafter\ifx\csname urlstyle\endcsname\relax
  \providecommand{\doi}[1]{doi: #1}\else
  \providecommand{\doi}{doi: \begingroup \urlstyle{rm}\Url}\fi

\bibitem[Dudek and Wajc(2011)]{DuWa}
A.~Dudek and D.~Wajc.
\newblock On the complexity of vertex-colouring edge-weightings.
\newblock \emph{Discrete Mathematics and Theoretical Computer Science},
  13:\penalty0 347--349, 2011.

\bibitem[Karonski et~al.(2004)Karonski, \L{}uczak, and Thomason]{KaLuTh}
M.~Karonski, T.~\L{}uczak, and A.~Thomason.
\newblock Edge weights and vertex colours.
\newblock \emph{J. Combinatorial Theory Ser. B}, 91:\penalty0 151--157, 2004.

\bibitem[Khatirinejad et~al.(2012)Khatirinejad, Naserasr, Newman, Seamone, and
  Stevens]{kha}
M.~Khatirinejad, R.~Naserasr, M.~Newman, B.~Seamone, and B.~Stevens.
\newblock Vertex-colouring edge-weightings with two edge weights.
\newblock \emph{Discrete Mathematics and Theoretical Computer Science},
  14:1:\penalty0 1--20, 2012.

\bibitem[Lu(2016)]{Lu}
H.~Lu.
\newblock Vertex-colouring edge-weighting of bipartite graphs with two edge
  weights.
\newblock \emph{Discrete Mathematics and Theoretical Computer Science},
  17:\penalty0 1--12, 2016.

\bibitem[Seamone()]{Sea}
B.~Seamone.
\newblock The 1-2-3 conjecture and related problems: a survey.
\newblock \emph{ArXiv: 1211.5122}.

\bibitem[Skowronek-Kazi\'{o}w(2017)]{skow}
J.~Skowronek-Kazi\'{o}w.
\newblock Graphs with multiplicative vertex-coloring 2-edge-weightings.
\newblock \emph{J. of Combinatorial Optimization}, 33:\penalty0 333--338, 2017.

\bibitem[Thomassen et~al.(2016)Thomassen, Wu, and Zhang]{TWZ}
C.~Thomassen, Y.~Wu, and C.-Q. Zhang.
\newblock The 3-flow conjecture, factors modulo k, and the 1-2-3-conjecture.
\newblock \emph{J. Combinatorial Theory Ser. B}, 121:\penalty0 308--325, 2016.

\end{thebibliography}
\label{sec:biblio}

\end{document}